\documentclass[sn-mathphys, Numbered]{sn-jnl}

\usepackage{graphicx}
\usepackage{amssymb}
\usepackage{extarrows}
\usepackage{a4wide}
\usepackage{tikz}
\usetikzlibrary{calc,matrix,arrows,snakes,backgrounds,shapes,automata,decorations,trees}
\usepackage{tkz-graph}
\usetikzlibrary{trees,positioning}
\usetikzlibrary{decorations.markings}
\tikzstyle directed=[postaction={decorate,decoration={markings, mark=at position .65 with {\arrow{stealth}}}}]
\tikzstyle reverse directed=[postaction={decorate,decoration={markings,mark=at position .65 with {\arrowreversed{stealth};}}}]
\numberwithin{figure}{section}
\usepackage{subfigure}

\usepackage{enumitem} 
\input xy
\xyoption{all}
\usepackage{multicol}
\usepackage{hyperref}
\hypersetup{colorlinks, linkcolor=blue,
filecolor=green,
urlcolor=blue,
citecolor=blue }

\usepackage{amsmath,amssymb,amsfonts}%
\usepackage{amsthm}%
\usepackage{mathrsfs}%
\usepackage[title]{appendix}%
\usepackage{xcolor}%
\usepackage{textcomp}%
\usepackage{manyfoot}

\theoremstyle{thmstyleone}%

\newtheorem{lemma}{Lemma}[section]
\newtheorem{corollary}[lemma]{Corollary}
\newtheorem{theorem}[lemma]{Theorem}
\newtheorem{proposition}[lemma]{Proposition}

\newtheorem{definition}[lemma]{Definition}

\newtheorem{example}[lemma]{Example}

\def\T+{{\mathbb T_d^+}}

\def\A{\mathcal{A}}
\def\span{\mathop{\rm span}}

\begin{document}

\title[Hilbert evolution algebras, weighted digraphs, and nilpotency]{Hilbert evolution algebras, weighted digraphs, and nilpotency}


\author[1]{\fnm{Paula} \sur{Cadavid}}\email{paula.cadavid@ufrpe.br}

\author[2]{\fnm{Pablo M.} \sur{Rodriguez}}\email{pablo@de.ufpe.br}

\author*[3]{\fnm{Sebastian J.} \sur{Vidal}}\email{svidal@unpata.edu.ar}

\affil[1]{\orgdiv{Departamento de Matem\'atica}, \orgname{Universidade Federal Rural de Pernambuco}, \orgaddress{\street{Rua Dom Manuel de Medeiros, s/n, Dois Irmãos}, \city{Pernambuco}, \postcode{52171-900}, \state{PE}, \country{Brazil}}}

\affil[2]{\orgname{Centro de Ci\^encias Exatas e da Natureza, Universidade Federal de Pernambuco}, \orgaddress{\street{Av. Prof. Moraes Rego - Cidade Universit\'aria}, \city{Recife}, \postcode{1235}, \state{PE}, \country{Brazil}}}

\affil[3]{\orgdiv{Departamento de Matem\'atica}, \orgname{ Facultad de Ingenier\'ia, Universidad Nacional de la Pa\-ta\-go\-nia ``San Juan Bosco''}, \orgaddress{\street{Ruta Prov. No. 1, S/N, Ciudad Universitaria, Km 4}, \city{Comodoro Rivadavia}, \postcode{9000}, \state{Chubut}, \country{Argentina}}}


\abstract{
Hilbert evolution algebras generalize evolution algebras through a framework of Hilbert spaces. In this work we focus on infinite-dimensional Hilbert evolution algebras and their representation through a suitably defined weighted digraph. By means of studying such a digraph we obtain new properties for these structures extending well-known results related to the nilpotency of finite dimensional evolution algebras. We show that differently from what happens for the  finite dimensional evolution algebras, the notions of nil and nilpotency are not equivalent for Hilbert evolution algebras. Furthermore, we exhibit necessary and sufficient conditions under which a given Hilbert evolution algebra is nil or nilpotent. Our approach includes  illustrative examples.}




\maketitle

\section{Introduction}
\subsection{An overview of evolution algebras} Since its formulation during the second half of 2000s, by \cite{tv,tian}, the Theory of Evolution Algebras has become a subject of increasing interest and active research. On one hand, this theory attracts great interest because of the many connections with other fields like Probability, Dynamical Systems and Graphs, among others. On other hand, these are non-associative algebras, which as a special class of genetic algebras, exhibit properties whose proofs are non-trivial and in many cases induce new ideas and open problems. An  evolution algebra is defined as follows.

\bigskip
\begin{definition} \cite[Definition 3]{tian}
\label{def:evolalgTian} Let  $\mathbb{K}$  be a field and let $\A:=(\A,\cdot\,)$ be a $\mathbb{K}$-algebra. We say that $\A$ is an evolution algebra if it admits a basis $S:=\{e_i\}_{i\in\Lambda}$, such that  
\begin{equation}
e_i \cdot e_i =\displaystyle \sum_{k\in\Lambda} c_{ik} e_k
\end{equation}
and $e_i \cdot e_j = 0$  for $i,j\in \Lambda$ such that $i\neq j$, where the scalars $c_{ik}$ are called structural constants.  
\end{definition}
\bigskip

Although it is not possible to include the many existing and relevant references about the subject in a few lines, we shall highlight some of them and we refer the reader to the references therein. Some of the main lines of research are focused on the study of different mathematical properties like the classification of evolution algebras \cite{Cabrera/Siles/Velasco}, the characterization of their derivation spaces \cite{PMP2,PMPY,camacho/gomez/omirov/turdibaev/LM2013,YPT}, and the analysis of related co-algebras \cite{Paniello-EC,Paniello-BECCP,Paniello-IEOGC}. Regarding connections with other fields, one can refer the reader to the following works \cite{PMP, PMPT, rozikov-velasco, Paniello-MEA,Mukhamedov-Qaralleh,ceballos}. The link with Probability Theory is through the concept of discrete-time Markov chain, this was suggested first by \cite[Chapter 4]{tian}, where many properties of Markov chains were translated in the language of evolution algebras. Such connection has been explored in different ways, for example, by \cite{casas-ladra-rozikov,Paniello-MEA,Mukhamedov-Qaralleh}, and also inspired new lines of research like in \cite{vidal-IJPAM}. The interplay between evolution algebras and dynamical systems is through the interpretation of the evolution operator of the algebra, which is defined as the 
endomorphism $L:\A\to \A$, such that $L(e_i)=e_i^2$, for any $i\in\Lambda$. In this direction, the reader may find interesting results in \cite{DOR,rozikov-velasco}. Evolution Algebras and Graphs is maybe one of the relations which more contribute to the study of the first ones. There are two ways to connect both subjects. The first one is, given a graph, to define an associated evolution algebra. One way to do it is to identify the structural constants $\{c_{ik}\}_{i,k\in\Lambda}$ with the elements of the adjacency matrix of the respective graph. In \cite{PMP,PMPT} the authors address one of the open problems suggested by \cite{tian}; namely, to find the connection between the evolution algebras associated to a graph and the one associated to the symmetric random walk on the same graph. The other direction in the connection of these algebras with graphs is, given an evolution algebra, to associate a (weighted) digraph. This is done in a natural way; namely, vertices represent generators and directed edges between two vertices, say $i$ and $k$, represent that  $c_{ik}\neq 0$. Moreover, a weighted graph is defined by adding the weight $c_{ik}$ to the directed edge from $i$ to $k$. The first in proposing such an identification were \cite{Elduque/Labra/2015,Elduque/Labra/2016}, where the authors used properties of the resulting graph to obtain new algebraic results on evolution algebras as well as some new and intuitive proofs of well-known results. Recently, such ideas were also explored by \cite{PMPY,YPT,Cabrera/Siles/Velasco, ceballos, Qaralleh-Mukhamedov}. The purpose of our work is to advance with this connection by identifying a suitable defined (weighted) digraph to a Hilbert evolution algebra, which was introduced by \cite{vidal-IJPAM} as a generalization of the concept of evolution algebras through a framework of Hilbert spaces.

\bigskip
\begin{definition}\label{def:evolalg} \cite[Definition 2]{vidal-IJPAM} Let $\A=(\A,\langle\cdot, \cdot\rangle)$ be a real or complex separable Hilbert space  which is provided with an algebra structure by the product $\cdot:\A\times\A\rightarrow \A$. We say that $\A:=(\A,\langle\cdot, \cdot\rangle,\cdot \,)$ is a separable Hilbert evolution algebra if it satisfies the following conditions:
\begin{enumerate}[label=(\roman*)]
\item \label{natural} There exists an orthonormal basis $\{e_i\}_{i\in\mathbb{N}}$ and scalars $\{c_{ik}\}_{i,k\in\mathbb{N}}$, such that 
\begin{equation}\label{eq:ea03}
e_i \cdot e_i =\displaystyle \sum_{k=1}^{\infty} c_{ik} e_k
\end{equation}
and
\begin{equation}\label{eq:ea04}
e_i \cdot e_j=0, \text{ if }i\neq j,
\end{equation} 
for any $i,j\in \mathbb{N}$. 
\smallskip
\item For any $v\in\A$, the left multiplications $L_v:\A\longrightarrow \A$ defined by $L_{v}(w):= v\cdot w$ are continuous in the metric topology induced by the inner product; i.e., there exist constants $M_v>0$ such that
\begin{equation}\label{eq:continuity condition}\nonumber
\|L_v(w)\|\leq M_v \|w\|,\text{ for all } w\in \A.
\end{equation}
\end{enumerate}
\end{definition}
\bigskip
A basis for $\A$, not necessarily orthonormal, satisfying \eqref{eq:ea03} and \eqref{eq:ea04} is called natural basis and the scalars $\{c_{ik}\}_{i,k\in\mathbb{N}}$ are called structural constants. In the following we will omit the word separable and simply say Hilbert evolution algebras, since we will only be working with separable Hilbert spaces.

The Definition \ref{def:evolalg} allows us to deal with a wide class of infinite dimensional spaces, as pointed out by \cite{vidal-IJPAM}, where the authors discuss a connection with discrete-time Markov chains with countable state space.

\subsection{Contribution and organization of the paper}

In \cite{vidal-SMJ} the authors study Hilbert evolution algebras associated to a graph and obtain results which extend to graphs with infinitely many vertices a theory developed by \cite{PMP,PMPT} for evolution algebras associated to finite graphs. In this work we shall explore the another direction; namely, given a Hilbert evolution algebra $\A$ we shall associate a suitably defined weighted digraph. Thanks to such a connection we will be able to extend to Hilbert evolution algebras well-known results related to the nilpotency of finite dimensional evolution algebras. We show that, in contrast to the finite dimensional case, the notions of nil and nilpotency are no longer equivalent for Hilbert evolution algebras. Furthermore, we exhibit necessary and sufficient conditions under which a given Hilbert evolution algebra is nil or nilpotent. We organize the paper as follows. Section \ref{preliminaries} is concerned with establishing some basic notation and definitions related to infinite weighed digraphs. In Section \ref{digraph_of_a_HEA} we introduce the associated weighted digraph of a given Hilbert evolution algebra, and we discuss some properties and examples. Section \ref{nilpotency} is devoted to extend to Hilbert evolution algebras the discussion of \cite{Elduque/Labra/2015} related to the nilpotency of evolution algebras. We also illustrate our results with some examples. 


\section{Preliminaries on infinite weighted digraphs}\label{preliminaries}

\subsection{Basic notation}
Let us start with some nomenclature of Graph Theory. A directed graph, or digraph, $G=(V,E)$ consists of a pair of countable sets: the vertex set $V$ and the directed edge set $E$, where a directed edge is an ordered pair of vertices. We say that $G$ is infinite if $V$ is an infinite countable set and in this case we take $V=\mathbb{N}$. A weighted digraph $G=(V,E,w)$ is a digraph endowed with a map $w:E\longrightarrow\mathbb{K}$ denoted by $w(i,j):=w_{ij}$. Note that we can always think a digraph as a weighted digraph with $w_{ij}=1$, if there is a directed edge from $i$ to $j$,   and $w_{ij}=0$,  otherwise.  
A  finite path (or a path for short) in a directed graph $G=(V, E)$ is a sequence $\{v_1,\ldots,v_n\}$ of vertices such that $(v_{i},v_{i+1}) \in E$  for $i\in\{1,\ldots ,n-1\}$.
Such a path determines a sequence of edges $\alpha_{1} =(v_{1}, v_{2}), \ldots, \alpha_{n-1}=(v_{n-1}, v_{n})$. The number of edges determined by a finite path is called the length of the path. A (finite) weak path is a sequence  $\{v_1,\ldots,v_n\}$ of vertices such that $(v_{i-1},v_i) \in E$ or $(v_{i},v_{i-1})\in E$  for any $i\in\{1,\ldots,n\}$. An oriented cycle is a non-empty (finite) path in which the only repeated vertices are the first and last vertices.  
If $v$ is reachable from $u$ through a path the distance between $u$ and $v$, denoted by $\mathrm{dist}(u,v)$, is the length of the shortest path connecting $u$ and $v$. A ray is a sequence $\{v_1,\ldots, v_n, \ldots\}$ formed by an infinite number of vertices such that $(v_i, v_{i+1}) \in E$ for all $i\in \mathbb{N}$.  


Given a digraph $G=(V,E)$, the first-generation descendants of $i_0\in V$ is the set given by $D^1(i_0):=\{k\in \mathbb{N}: (i_0,k)\in E \}$. In addition, given a subset $U\subset V$, we let $D^1(U):=\{k\in \mathbb{N}: k \in D^1(i) \text{ for some } i \in U\}$. Similarly, we say that $j$ is a second-generation descendant of $i_0$, if $j\in D^1(k)$ for some $k\in D^1(i_0)$. Therefore, 
$$D^2(i_0):= \bigcup_{k\in D^1(i_0)}D^1(k). $$
Recursively, we define the m$th-$generation descendants of $i_0$ as 
$$D^m(i_0):= \bigcup_{k\in D^{m-1}(i_0)}D^1(k). $$
Finally, the set of descendants of $i_0$ is defined as 
$$D(i_0):= \bigcup_{m\in \mathbb{N}}D^m(i_0). $$
By convention we put $D^0(i_0):=\{i_0\}$ and $D^0(U):=U$. We have similar definitions for the ascendants of $i_0$, which will be denoted by $A^1(i_0):=\{k\in \mathbb{N}: (k, i_0)\in E \}$, $A^1(U)$, $A^m(i_0)$ and $A^{0}(i_0)$ for $i_0\in V$ and $U\subset V$. In this work we shall use the notation $G_i$ to represent, for any $i\in\mathbb{N}$, the graph $(V_i, E_i)$ with $V_i:=D(i)$ and $E_i=\{(k, \ell): k, \ell \in V_i$ and $(k, \ell)\in E\}$. Moreover, we define the depth of $G_i$ as $\delta(G_i):=\sup\{\mathrm{dist}(i,u):u\in D(i)\}$, and we let $\delta(G_i):=\infty$ if the supremum  does not exist. 
Also, given a vertex $i\in V$, we define $\deg^+(i)$ as the cardinality of the set $D^{1}(i)$  and $\deg^-(i)$ as the cardinality of $A^{1}(i)$. If we put $a_{ik}:=1$ whenever $(i,k)\in E$ and $a_{ik}:=0$ otherwise, then 
$$\deg^+(i)=\sum_{k=1}^\infty a_{ik}
\quad \text{ and }\quad \deg^-(i)=\sum_{k=1}^\infty a_{ki}.$$
With this we define $\deg(i):=\deg^+(i)+\deg^-(i)$ and we say that the $G$ is locally finite if $\deg^{+}(i)<\infty$ for all $i\in V$.

\subsection{Operators associated to an infinite digraph}

Following \cite{Fujii, Exner}, we introduce the weighted adjacency operator and the adjacency operator associated to a given infinite weighted digraph, denoted by $G=(V,E,w)$.
In this framework, the adjacency matrix of $G$ is denoted by $(a_{ik})_{i,k\in\mathbb{N}}$ and the weighted adjacency matrix of $G$ is denoted by $(\omega_{ik})_{i,k\in\mathbb{N}}$. Both can be interpreted as operators $A$, $\Omega$ densely defined in $\ell^2(\mathbb{N})$
under appropriate assumptions. 
Let us analyze this.
If $D_0$ is the dense subset of $\ell^2(\mathbb{N})$ formed by those sequences with only a finite number of non-zero components, we can define the operators $\Omega_0,\Gamma_0:\ell^2(\mathbb{N})\longrightarrow \ell^2(\mathbb{N})$ by
\begin{equation*}
\Omega_0(\delta_i):=\sum_{k=1}^{\infty}\omega_{ik} \delta_k
\quad\text{ and } \quad
\Gamma_0(\delta_i):=\sum_{k=1}^{\infty}\overline{w_{ik}} \delta_k
\end{equation*} 
for all $i\in\mathbb{N}$, where we denote by $\delta_k:=\{\delta_{ik}\}_{i\in\mathbb{N}}$ the standard orthonormal basis of $\ell^2(\mathbb{N})$, 
and we extend by linearity to all $D_0$, giving the domains $D(\Omega_0)=D(\Gamma_0)=D_0$. 
That is, $\Omega_0$ and $\Gamma_0$ are densely defined.
Now, we define the weighted adjacency operator $\Omega:\ell^2(\mathbb{N})\longrightarrow\ell^2(\mathbb{N})$ for  $v=\sum_{i=1}^{\infty}v_i\delta_i$ by 
\begin{equation}
    \Omega(v):=\sum_{k=1}^\infty \sum_{i=1}^\infty v_i \omega_{ik}\delta_k.\end{equation}
The corresponding domain is given by
\begin{equation} \nonumber
D(\Omega):=\biggl \{v=\sum_{i=1}^{\infty}v_i\delta_i\in \ell^2(\mathbb{N})\,:\, \sum_{k=1}^{\infty}\biggl|\sum_{i=1}^{\infty}
v_i \omega_{ik} \biggr|^2 <\infty \biggr\}.
\end{equation}
Also, we define the operator $\Gamma:\ell^2(\mathbb{N})\longrightarrow\ell^2(\mathbb{N})$ by
\begin{equation}\nonumber
\Gamma(v):=\sum_{k=1}^\infty \sum_{i=1}^\infty v_i\overline{\omega_{ki}} \delta_k
\end{equation}
with domain
\begin{equation}\nonumber
D(\Gamma):=\biggl \{v=\sum_{i=1}^{\infty}v_i\delta_i\in \ell^2(\mathbb{N})\,:\, \sum_{k=1}^{\infty}\biggl|\sum_{i=1}^{\infty}v_i \overline{\omega_{ki}} \biggr|^2 <\infty \biggr\}.
\end{equation}
Note that $D(\Omega)$ and $D(\Gamma)$ are linear subspaces of $\ell^2(\mathbb{N})$.
Similar considerations apply to the definition of the adjacency operator $A:\ell^2(\mathbb{N})\longrightarrow\ell^2(\mathbb{N})$ by
\begin{equation}\nonumber
A(v):=\sum_{k=1}^\infty \sum_{i=1}^\infty v_i a_{ik}\delta_k,   \text{ where } v=\sum_{i=1}^{\infty}v_i\delta_i
\end{equation}
and the operator $B:\ell^2(\mathbb{N})\longrightarrow\ell^2(\mathbb{N})$ by
\begin{equation}\nonumber
B(v):=\sum_{k=1}^\infty \sum_{i=1}^\infty v_i\overline{a_{ki}} \delta_k .
\end{equation}
The corresponding domains are given by
\begin{equation}\nonumber
D(A):=\biggl \{v=\sum_{i=1}^{\infty}v_i\delta_i\in \ell^2(\mathbb{N})\,:\, \sum_{k=1}^{\infty}\biggl|\sum_{i=1}^{\infty}
v_i a_{ik} \biggr|^2 <\infty \biggr\}
\end{equation}	
and
\begin{equation}\nonumber
D(B):=\biggl \{v=\sum_{i=1}^{\infty}v_i\delta_i\in \ell^2(\mathbb{N})\,:\, \sum_{k=1}^{\infty}\biggl|\sum_{i=1}^{\infty}v_i \overline{a_{ki}} \biggr|^2 <\infty \biggr\}.
\end{equation}
Note that $A$ and $B$ can be obtained by taking $\omega_{ik}=a_{ik}$ in $\Omega$ and $\Gamma$, respectively.

\bigskip
\begin{proposition}
\label{prop:densely def condition}
Let $G=(V,E,\omega)$ be an infinite weighted digraph.
Suppose that
\begin{equation}
\label{eq:densely def condition 01}
\sum_{k=1}^{\infty} |\omega_{ik}|^2 <\infty \,,\,\text{for all } i\in\mathbb{N}.
\end{equation}
Then the weighted adjacency operator $\Omega$ is densely defined.
Furthermore, $\Omega^*\subset \Omega_0^*=\Gamma$ and $\Gamma$ is closed.
Analogously, if
\begin{equation}\label{eq:densely def condition 02}
\sum_{i=1}^\infty |\omega_{ik}|^2 <\infty 
\,,\,\text{for all } k\in\mathbb{N}
\end{equation}
then $\Gamma$ is densely defined.
Furthermore, $\Gamma^*\subset \Gamma_0^*=\Omega$ and $\Omega$ is closed. For the  adjacency operator $A$ and for the operator  $B$ we have similar conclusions.
\end{proposition}
\begin{proof}
By definition of the domain we have $\delta_i\in D(\Omega)$ if and only if 
\begin{equation}\nonumber
\|\Omega_0 (\delta_i)\|^2=\sum_{k=1}^\infty |\omega_{ik}|^2<\infty.
\end{equation}
Thus, by equation \eqref{eq:densely def condition 01},  $\Omega$ 
is densely defined, because $D_0$ is a dense subset of $\ell^2(\mathbb{N})$. 
Let us prove that $\Gamma\subset \Omega_0^*$.
Consider $v=\sum_{j=1}^\infty v_j\delta_j\in D(\Gamma)$, then
\begin{equation}\nonumber
\langle \Omega_0(\delta_i), v\rangle
=\sum_{k=1}^\infty\sum_{j=1}^\infty\omega_{ik}\overline{v_j}\langle \delta_k, \delta_j\rangle
=\sum_{j=1}^\infty \omega_{ij}\overline{v_j}= \sum_{k=1}^\infty\sum_{j=1}^\infty \overline{\overline{\omega_{kj}} v_j}\langle \delta_i, \delta_k\rangle = \langle \delta_i, \Gamma(v)\rangle, \text{ for } i\in\mathbb{N}.
\end{equation}
Thus $\Omega^*_0(v)=\Gamma(v)$, for all $v\in D(\Gamma)$, i.e., $\Gamma\subset \Omega_0^*$.
Now, for $\Omega_0^*=\Gamma$ we must prove that $D(\Omega_0^*)\subset D(\Gamma)$.
Let $v=\sum_{j=1}^\infty v_j\delta_j\in D(\Omega_0^*)$ and recall that $\delta_k\in D(\Omega)$, then
\begin{equation}\nonumber
\langle \delta_k, \Omega_0^*(v)\rangle = \langle \Omega_0(\delta_k),v\rangle 
=\sum_{k=1}^\infty \overline{v_i} \omega_{ki}.
\end{equation}
This implies
\begin{equation}\nonumber
\sum_{k=1}^\infty \biggl|\sum_{i=1}^\infty v_i\overline{\omega_{ki}} \biggr|^2
=\sum_{k=1}^\infty \biggl|\sum_{i=1}^\infty \overline{v_i} \omega_{ki}\biggr|^2
=\sum_{k=1}^\infty |\langle \delta_k, \Omega_0^*(v)\rangle|^2
=\|\Omega_0^* (v)\|^2<\infty,
\end{equation}
that is $v\in D(\Gamma)$. 
It follows that $\Omega^*\subset \Omega_0^*=\Gamma$.
But $\Omega$ is densely defined, thus the adjoint $\Omega_0^*=\Gamma$ is a closed operator \cite[Proposition 1.6]{Schmudgen}. The same arguments are valid for $\Gamma$, i.e., $\delta_k\in D(\Gamma)$ if and only if 
\begin{equation}\nonumber
\|\Gamma_0(\delta_k)\|^2=\sum_{i=1}^\infty |\overline{\omega_{ik}}|^2
=\sum_{i=1}^\infty |\omega_{ik}|^2<\infty.
\end{equation}
 Which implies that $\Gamma$ is densely defined.
Also $\langle \Gamma_0(\delta_j), v\rangle=\langle \delta_j, \Omega(v)\rangle$, for all $j\in\mathbb{N}$.
Thus $\Gamma^*_0(v)=\Omega(v)$, for all $v\in D(A)$, i.e., $\Omega\subset \Gamma_0^*$.
To show that $\Gamma_0^*=\Omega$ we can prove that $D(\Gamma_0^*)\subset D(\Omega)$ using that for  $v=\sum_{j=1}^\infty v_j\delta_j$ 
\begin{equation}\nonumber
\langle \delta_k,\Gamma_0^*(v)\rangle = \langle \Gamma_0(\delta_k),v\rangle 
=\sum_{i=1}^\infty\overline{\omega_{ik}}\,\overline{v_i}
\end{equation}
and
\begin{equation}\nonumber
\sum_{k=1}^\infty\biggl|\sum_{i=1}^{\infty}\omega_{ik} v_i \biggr|^2
=\sum_{k=1}^\infty \biggl|\sum_{i=1}^\infty\overline{\omega_{ik}}\,\overline{v_i} \biggr|^2
=\sum_{k=1}^\infty|\langle \delta_k, \Gamma_0^*(v)\rangle|^2
=\|\Gamma_0^* (v)\|^2<\infty .
\end{equation}
Again, it follows that $\Gamma^*\subset \Gamma_0^*=\Omega$.
And $\Gamma$ is densely defined, thus the adjoint $\Gamma_0^*=\Omega$ is a closed operator.
\end{proof}
\bigskip
\begin{corollary}
Let $G=(V,E)$ be an infinite digraph.
If $\deg^+(i)<\infty$, for all $i\in\mathbb{N}$, then the adjacency operator $A$ is densely defined, 
$A^*\subset A_0^*=B$ and $B$ is closed.
If $\deg^-(i)<\infty$, for all $i\in\mathbb{N}$, then $B$ is densely defined, 
$B^*\subset B_0^*=A$ and $A$ is closed.
\end{corollary}
\begin{proof}
Let $i\in \mathbb{N}$. 
If $\deg^+(i)<\infty$ then $a_{ik}\neq 0$, only for finite values of $k$. Hence \eqref{eq:densely def condition 01} with $\omega_{ik}=a_{ik}$, is satisfied and we can apply 
the Proposition \ref{prop:densely def condition} to obtain that $A$ is densely defined, 
$A^*\subset A_0^*=B$ and $B$ is closed.
The proof of other claim is similar, using equation \eqref{eq:densely def condition 02}.
\end{proof}

For the continuity of the adjacency operator we need stronger assumptions.
\bigskip

\begin{proposition}\label{prop:bounded condition 0}
Let $G=(V,E,\omega)$ be an infinite weighted digraph.
Suppose that one of the following conditions is satisfied:
\begin{enumerate}[label=(\roman*)]
\item \label{item:condition continuity C 0}
$\displaystyle M:=\left(\sum_{k=1}^{\infty}\sum_{i=1}^{\infty}|\omega_{ik}|^2\right)^{1/2}<\infty.$ \vspace{0.2cm}
		
\item There exist constants $\alpha_k, \beta_i>0$, $i,k\in \mathbb{N}$ and $M_1, M_2>0$ such that
\begin{equation}\label{eq:schur test 0}
\begin{array}{l}
\displaystyle\sum_{k=1}^{\infty}|\omega_{ik}|\alpha_k\leq M_1\beta_i, \quad\text{ for all } i\in \mathbb{N}, \\[1.1ex]
\displaystyle\sum_{i=1}^{\infty}|\omega_{ik}|\beta_i\leq M_2\alpha_k,
\quad\text{ for all } k\in \mathbb{N}.
\end{array}
\end{equation}
\end{enumerate}
Then the weighted adjacency operator $\Omega$ is bounded. 
In the first case we have $\|\Omega\|\leq M$ and in the second case we have $\|\Omega\|\leq (M_1M_2)^{1/2}$.
\end{proposition}

\begin{proof}
The following argument is very similar to \cite[Proposition 3]{vidal-IJPAM}, but with modifications corresponding to our framework of infinite weighted digraphs.
First, note that the Condition \ref{item:condition continuity C 0} implies the Equations
 \eqref{eq:densely def condition 01} and \eqref{eq:densely def condition 02} are valid. 
Thus, by Proposition \ref{prop:densely def condition}, the operator $\Omega$ is densely defined and closed. 
Let $v=\sum_{i=1}^\infty v_i\delta_i\in D(\Omega)$, by Hölder inequality in $\ell^2$ we have
\begin{equation}\nonumber
\biggl|\sum_{i=1}^\infty\omega_{ik}v_i\biggr| \leq \sum_{i=1}^\infty| \omega_{ik}v_i| \leq 
\left( \sum_{i=1}^\infty|\omega_{ik}|^2\right)^{1/2}\left(\sum_{i=1}^\infty|v_{i}|^2\right)^{1/2}.
\end{equation}
Then,
\begin{equation}\nonumber
\|\Omega(v)\|^2 =\sum_{k=1}^\infty\biggl|\sum_{i=1}^\infty\omega_{ik}v_i\biggr|^2
\leq \sum_{k=1}^\infty \left( \sum_{i=1}^\infty|\omega_{ik}|^2\right)\left(\sum_{i=1}^\infty|v_{i}|^2\right)
=M^2\|v\|^2.
\end{equation}
Hence $\Omega$ is bounded and $\|\Omega\|\leq M$. Now let us assume that \eqref{eq:schur test 0} is satisfied. By a similar argument used to prove the Schur Test \cite[Section 45]{Halmos}, we have that 
\begin{equation*}
\begin{array}{rl}
\|\Omega(v)\|^2 
&\displaystyle=\sum_{k=1}^{\infty} \biggl|\sum_{i=1}^{\infty} v_i \omega_{ik}\biggr|^2 \\
&\displaystyle\leq \sum_{k=1}^{\infty} \biggl|\sum_{i=1}^{\infty} |v_i| |\omega_{ik}|\biggr|^2 \\
&\displaystyle=\sum_{k=1}^{\infty} \biggl|\sum_{i=1}^{\infty}  \left(\sqrt{|\omega_{ik}|}\sqrt {\beta_i}\right) \left(\frac{\sqrt{|\omega_{ik}|}|v_i|}{\sqrt {\beta_i}}\right)\biggr|^2 \\
&\displaystyle\leq \sum_{k=1}^{\infty} \left(\sum_{i=1}^{\infty} |\omega_{ik}| \beta_i\right) \left(\sum_{i=1}^{\infty} \frac{|\omega_{ik}| |v_i|^2}{\beta_i} \right) \\
& \displaystyle \leq\sum_{k=1}^{\infty} M_2\alpha_k \left(\sum_{i=1}^{\infty} \frac{|\omega_{ik}||v_i|^2}{\beta_i}\right) \\
&\displaystyle =M_2 \sum_{i=1}^{\infty}  \frac{|v_i|^2}{\beta_i}\left(\sum_{k=1}^{\infty}|\omega_{ik}|\alpha_k\right)\\
&\displaystyle \leq M_1M_2 \sum_{i=1}^{\infty}  |v_i|^2
=M_1M_2 \|v\|^2.
\end{array}
\end{equation*}
That is $\Omega$ a bounded linear operator, with $\|\Omega\|\leq (M_1 M_2)^{1/2}$. 
 \end{proof}

\bigskip
\begin{corollary}
Let $G=(V,E)$ be a infinite locally-finite digraph. Then the adjacency operator $A$ is bounded
if and only if the Condition \ref{item:condition continuity C 0} of Proposition 
\ref{prop:bounded condition 0} is satisfied.
In this case $\|A\|\leq M$.
\end{corollary}
\begin{proof}
One implication is the Proposition \ref{prop:bounded condition 0}.
Suppose that $A$ is bounded, then there exists $K>0$ such that $\|A(v)\|\leq K\|v\|$ for all $v\in D(A)$.
Since $G$ is locally finite then  for fixed $k\in \mathbb{N}$ we have $a_{kj}\neq 0$ only for a finite number of $j$'s, specifically this is occurs if $j\in D^1(k)$.
Hence
\begin{equation}\nonumber
M=\sum_{i=1}^\infty\sum_{j=1}^{\infty} |a_{ij}|^2
=\sum_{j\in D^1(i)}\sum_{i=1}^{\infty} a_{ij}
=\sum_{j\in D^1(i)}\|A(\delta_j)\|^2
\leq \sum_{j\in D^1(i)} K^2 <\infty.
\end{equation}

\end{proof}

\section{The weighted digraph of a Hilbert Evolution Algebra}\label{digraph_of_a_HEA}

\subsection{Definitions, properties and examples} 


\begin{definition} Let $\A$ be a Hilbert evolution algebra with a natural orthonormal basis $\mathcal{B}=\{e_i\}_{i\in\mathbb{N}}$  and  structural constants $\{c_{ik}\}_{i,k\in\mathbb{N}}$. The weighted digraph $G(\A,\mathcal{B}):=(V,E,\omega)$ with  $V=\mathbb{N}$, $E:=\{(i,k)\in\mathbb{N} \times \mathbb{N}:  c_{ik}\not= 0 \}$ and $\omega: E\rightarrow \mathbb{K}$ given by $\omega_{ik}:=c_{ik}$ is called the weighted digraph associated to the Hilbert evolution algebra $\A$ relative to the basis $\mathcal{B}$.
\end{definition}
Now, we want to define the corresponding adjacency and weighted adjacency operators, $A$ and $\Omega$.
Thus, we have to study the conditions on the coefficients to obtain densely defined operators.
Recall that the adjacency matrix of the graph $G(\A, \mathcal{B})$ has been used in \cite{Cabrera/Siles/Velasco,Elduque/Labra/2015,Elduque/Labra/2016} to study some properties for, mainly finite-dimensional, evolution algebras. Here we explore a bit more the same idea to obtain some properties for Hilbert evolution algebras. 
First, we translate the previous results of weighted digraphs to this context, to understand when the associated operators are densely defined or closed.
\bigskip
\begin{proposition}
\label{prop:adjacency and weighted densely defined and closed}
Let $\A$ be a Hilbert evolution algebra  with a  natural orthonormal basis $\mathcal{B}=\{e_i\}_{i\in\mathbb{N}}$ and $G(\A,\mathcal{B})$ the associated weighted digraph.
Then:
\begin{enumerate}[label=(\roman*)]
\item  \label{prop:adjacency and weighted densely defined and closed_i}The weighted adjacency operator $\Omega:\ell^2(\mathbb{N})\longrightarrow\ell^2(\mathbb{N})$ is densely defined.
\item \label{prop:adjacency and weighted densely defined and closed_ii} If the structure constants $\{\omega_{ik}\}_{i,k\in\mathbb{N}}$ 
satisfy the Equation \eqref{eq:densely def condition 02},
we have that $\Omega$ is closed.
\item  \label{prop:adjacency and weighted densely defined and closed_iii}If the adjacency constants $\{a_{ik}\}_{i,k\in\mathbb{N}}$ satisfy
\begin{equation}\nonumber
\sum_{k=1}^{\infty} |a_{ik}|^2 <\infty \,,\,\text{for all } i\in\mathbb{N}
\end{equation}
the adjacency operator $A:\ell^2(\mathbb{N})\longrightarrow\ell^2(\mathbb{N})$ is densely defined.
\item \label{prop:adjacency and weighted densely defined and closed_iv}If 
\begin{equation}\nonumber
\sum_{i=1}^{\infty} |a_{ik}|^2 <\infty \,,\,\text{for all } k\in\mathbb{N}
\end{equation}
we have that $A$ is closed.
\end{enumerate}
\end{proposition}
\begin{proof}
To prove \ref{prop:adjacency and weighted densely defined and closed_i} note that $\A$ is a Hilbert evolution algebra, thus $e_i\in\A$, for all $i\in\mathbb N$.
This implies
\begin{equation}\nonumber
\|e_i^2\|^2= \sum_{k=1}^{\infty} |\omega_{ik}|^2 <\infty \,,\,\text{for all } i\in\mathbb{N}.
\end{equation}
Then, by Proposition \ref{prop:densely def condition}, $\Omega$ is densely defined.
To show \ref{prop:adjacency and weighted densely defined and closed_ii} apply directly Proposition \ref{prop:densely def condition} to obtain that $\Omega$ is closed. Finally, to see \ref{prop:adjacency and weighted densely defined and closed_iii} and \ref{prop:adjacency and weighted densely defined and closed_iv} use the Proposition \ref{prop:densely def condition} again, with $\omega_{ik}=a_{ik}$.
\end{proof}

Note that if the operator $\Omega:\ell^2(\mathbb{N})\longrightarrow\ell^2(\mathbb{N})$ is densely defined, it follows that it can be represented by an infinite matrix $\{\omega_{ij}\}_{i,j\in \mathbb{N}}$.
The same is valid for the adjacency operator $A:\ell^2(\mathbb{N})\longrightarrow\ell^2(\mathbb{N})$, if it is densely defined it can be represented by an infinite matrix $\{a_{ij}\}_{i,j\in \mathbb{N}}$.

\bigskip
\begin{proposition}\label{prop:operators-closed}
Let $\A$ be a Hilbert evolution algebra  with a  natural orthonormal basis $\mathcal{B}=\{e_i\}_{i\in\mathbb{N}}$ and $G(\A,\mathcal{B})$ the associated weighted digraph.
If the structure constants $\{\omega_{ik}\}_{i,k\in\mathbb{N}}$ satisfy one of the following conditions:
\begin{enumerate}[label=(\roman*)]
\item \label{item:condition continuity C}
$\displaystyle M:=\left(\sum_{k=1}^{\infty}\sum_{i=1}^{\infty}|\omega_{ik}|^2\right)^{1/2}<\infty.$ \vspace{0.2cm}
  	
\item There exist constants $\alpha_k, \beta_i>0$, $i,k\in \mathbb{N}$ and $M_1, M_2>0$ such that
\begin{equation}\nonumber \label{eq:schur test}
\begin{array}{l}
\displaystyle\sum_{k=1}^{\infty}|\omega_{ik}|\alpha_k\leq M_1\beta_i, \quad\text{ for all } i\in \mathbb{N}, \\[1.1ex]
\displaystyle\sum_{i=1}^{\infty}|\omega_{ik}|\beta_i\leq M_2\alpha_k,
\quad\text{ for all } k\in \mathbb{N},
\end{array}
\end{equation}
\end{enumerate}
then the weighted adjacency operator $\Omega$ is bounded. 
In the first case we have $\|\Omega\|\leq M$ and in the second case we have $\|\Omega\|\leq (M_1M_2)^{1/2}$.
\end{proposition}
\begin{proof}
Apply Proposition \ref{prop:bounded condition 0}.
\end{proof}

All this allow us to think in concrete examples.
\bigskip
\begin{example}\label{ex:tree-HEA}
Fix $r\in\mathbb{N}$, with $r\geq 2$, denote $[r]:=\{1,\ldots,r\}$ and consider the countable set $$\Lambda:=\bigcup_{n=1}^{\infty}\left\{\{1\}\times[r]^{n-1}\right\},$$
where $[r]^n$ denotes the $n$-ary Cartesian power of the set $[r]$. Let $\A$ be the Hilbert evolution algebra with natural orthonormal basis $\mathcal{B}=\{e_i\}_{i\in \Lambda}$, and multiplication given by:
$$e_{1}^2=\sum_{j=1}^{r}c_{1,1j}e_{1j},$$
and, for any $k\geq 2$ and $1, i_2, \ldots, i_k\in [r]$ we let $$e_{1 i_2\ldots i_k}^2=\sum_{j=1}^{r}c_{1 i_2\ldots i_k,1 i_2\ldots i_k j}e_{1 i_2\ldots i_k j}.$$
In other words, we assume that the structural constants satisfy
$$c_{i,ij}\left\{
\begin{array}{cl}
    \neq 0 &\text{ if }i\in \{1\}\times[r]^{n-1},\text{ and }j\in [r]
    \text{ for some }n\in\mathbb{N},\\[.2cm]
    =0 &\text{ other case.} 
\end{array}\right.
$$
In this case, $G(\A,\mathcal{B})$ is a weighted infinite rooted $r$-ary tree; i.e., one vertex can be designated as the root and every edge is directed away from the root. Here the root is vertex $1$. See Figure \ref{fig:tree-HEA} for details. Moreover, since $G(\A,\mathcal{B})$ is locally finite its adjacency operator is densely defined and both its adjacency and the weighted adjacency operators are closed, see Proposition \ref{prop:adjacency and weighted densely defined and closed}.    
\end{example}

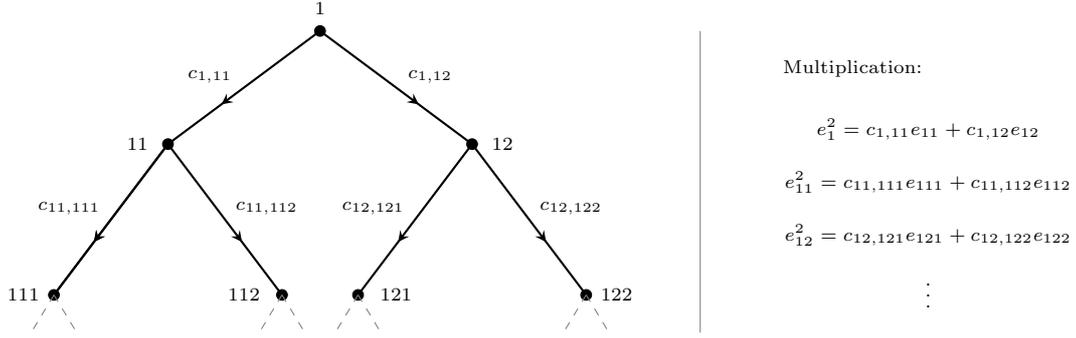
\begin{figure}
    \centering
\begin{tikzpicture}
\filldraw [black] (0,0) circle (2pt);
\node at (0,0.3) {\footnotesize $1$};
\node at (-1.45,-0.6) {\footnotesize $c_{1,11}$};
\node at (1.45,-0.6) {\footnotesize $c_{1,12}$};
\filldraw [black] (-2,-1.5) circle (2pt);
\node at (-2.4,-1.5) {\footnotesize $11$};
\node at (-3.3,-2.35) {\footnotesize $c_{11,111}$};
\node at (-0.7,-2.35) {\footnotesize $c_{11,112}$};
\filldraw [black] (2,-1.5) circle (2pt);
\node at (2.4,-1.5) {\footnotesize $12$};
\node at (3.3,-2.35) {\footnotesize $c_{12,122}$};
\node at (0.7,-2.35) {\footnotesize $c_{12,121}$};
\filldraw [black] (-3.5,-3.5) circle (2pt);
\node at (-3.9,-3.5) {\footnotesize $111$};
\filldraw [black] (-0.5,-3.5) circle (2pt);
\node at (-1,-3.5) {\footnotesize $112$};
\filldraw [black] (3.5,-3.5) circle (2pt);
\node at (1,-3.5) {\footnotesize $121$};
\filldraw [black] (0.5,-3.5) circle (2pt);
\node at (3.9,-3.5) {\footnotesize $122$};

\draw[thick,directed] (0,0) to (-2,-1.5);
\draw[thick,directed] (0,0) to (2,-1.5);
\draw[thick,directed] (-2,-1.5) to (-3.5,-3.5);
\draw[thick,directed] (-2,-1.5) to (-0.5,-3.5);
\draw[thick,directed] (2,-1.5) to (3.5,-3.5);
\draw[thick,directed] (2,-1.5) to (0.5,-3.5);
\draw[thick,directed] (-2,-1.5) to (-3.5,-3.5);
\draw[dashed,gray] (-3.5,-3.5) to (-3.8,-4);
\draw[dashed,gray] (-3.5,-3.5) to (-3.2,-4);
\draw[dashed,gray] (3.5,-3.5) to (3.8,-4);
\draw[dashed,gray] (3.5,-3.5) to (3.2,-4);
\draw[dashed,gray] (0.5,-3.5) to (0.8,-4);
\draw[dashed,gray] (0.5,-3.5) to (0.2,-4);
\draw[dashed,gray] (-0.5,-3.5) to (-0.8,-4);
\draw[dashed,gray] (-0.5,-3.5) to (-0.2,-4);

\draw[gray] (5,0) to (5,-4);

\node at (7,-0.5) {\footnotesize Multiplication:};

\node at (8,-1.3) {\footnotesize $e_1^2=c_{1,11}e_{11}+c_{1,12}e_{12}$};
\node at (8,-2) {\footnotesize $e_{11}^2=c_{11,111}e_{111}+c_{11,112}e_{112}$};
\node at (8,-2.7) {\footnotesize $e_{12}^2=c_{12,121}e_{121}+c_{12,122}e_{122}$};
\node at (8,-3.4) {\footnotesize $\vdots$};
\end{tikzpicture}

    \caption{The weighted digraph associated to the algebra $\A$ of Example \ref{ex:tree-HEA}, for $r=2$. Given the algebra multiplication (right side) the associated weighted digraph is a rooted $r$-ary tree (left side).}
    \label{fig:tree-HEA}
\end{figure}

\bigskip
\begin{example}\label{exa:nonLF}
Let $\A$ be the Hilbert evolution algebra with  natural orthonormal basis $\mathcal{B}=\{e_i\}_{i\in \mathbb{N}}$, and multiplication given by:
$$e_1^2=\displaystyle \sum_{j=2}^{\infty}c_{1j}e_j \quad  \text{ and } \quad e_i^2= e_{i+1}, \quad  \text{ for } i\geq 2,$$ 
where we assume that $c_{1j}\neq 0$ for any $j\neq 1$, and   $\sum_{j=1}^{\infty}c_{1j}=1$. Thus defined, the structural constants are transition probabilities for a discrete-time Markov chain with infinitely many states so this is an example of a Markov Hilbert evolution algebra. Moreover, note that  $G(\A,\mathcal{B})$ is a non-locally finite weighted digraph because $\deg^+(1)=\infty$. See Figure \ref{fig:nonLF} for details. 

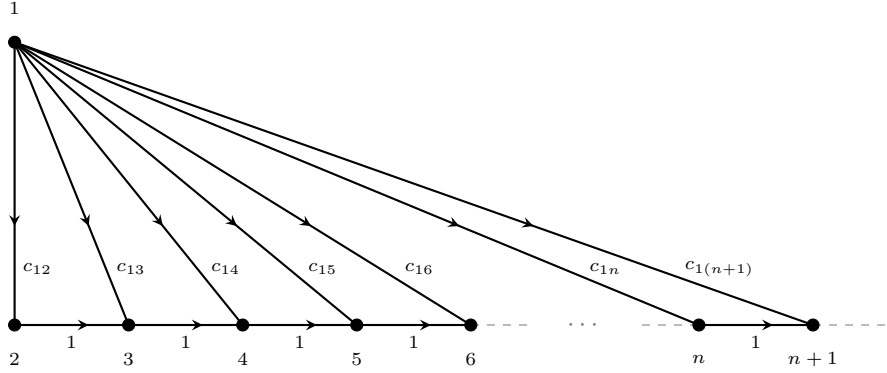
\begin{figure}[h!]
    \centering
    \begin{tikzpicture}[scale=1.5]
    \node at (5,0) {\color{gray}$\cdots$};

\draw[dashed,gray] (4,0) -- (4.5,0);
\draw[dashed,gray] (5.5,0) -- (6,0);
\draw[thick,directed] (0,0) -- (1,0);
\draw[thick,directed] (1,0) -- (2,0);
\draw[thick,directed] (2,0) -- (3,0);
\draw[thick,directed] (3,0) -- (4,0);
\draw[thick,directed] (6,0) -- (7,0);
\draw[dashed,gray] (7,0) -- (7.7,0);
\draw[thick,directed] (0,2.5) -- (0,0);
\draw[thick,directed] (0,2.5) -- (1,0);
\draw[thick,directed] (0,2.5) -- (2,0);
\draw[thick,directed] (0,2.5) -- (3,0);
\draw[thick,directed] (0,2.5) -- (4,0);
\draw[thick,directed] (0,2.5) -- (6,0);
\draw[thick,directed] (0,2.5) -- (7,0);
\filldraw [black] (0,0) circle (1.5pt);    
\node at (0,-0.3) {\footnotesize $2$};
\node at (0.5,-0.15) {\footnotesize $1$};
\node at (1.5,-0.15) {\footnotesize $1$};
\node at (2.5,-0.15) {\footnotesize $1$};
\node at (3.5,-0.15) {\footnotesize $1$};
\node at (6.5,-0.15) {\footnotesize $1$};
\node at (0.2,0.5) {\footnotesize $c_{12}$};
\node at (1.02,0.5) {\footnotesize $c_{13}$};     
\node at (1.85,0.5) {\footnotesize $c_{14}$};     
\node at (2.7,0.5) {\footnotesize $c_{15}$};     
\node at (3.55,0.5) {\footnotesize $c_{16}$};     
\node at (5.18,0.5) {\footnotesize $c_{1n}$};     
\node at (6.18,0.5) {\footnotesize $c_{1(n+1)}$};     
\filldraw [black] (1,0) circle (1.5pt);     \node at (1,-0.3) {\footnotesize $3$}; 
\filldraw [black] (2,0) circle (1.5pt);     \node at (2,-0.3) {\footnotesize $4$}; 
\filldraw [black] (3,0) circle (1.5pt);     \node at (3,-0.3) {\footnotesize $5$}; 
\filldraw [black] (4,0) circle (1.5pt);     \node at (4,-0.3) {\footnotesize $6$}; 
\filldraw [black] (6,0) circle (1.5pt);     \node at (6,-0.3) {\footnotesize $n$}; 
\filldraw [black] (7,0) circle (1.5pt);     \node at (7,-0.3) {\footnotesize $n+1$}; 
\filldraw [black] (0,2.5) circle (1.5pt);     \node at (0,2.8) {\footnotesize $1$}; 

    \end{tikzpicture}
    \caption{ The weighted digraph associated to the Hilbert evolution algebra
 from Example \ref{exa:nonLF}. Given the algebra multiplication the associated digraph is a non-locally finite graph with infinitely many vertices.}
    \label{fig:nonLF}
\end{figure}
\end{example}

The use of an associated digraph to study evolution algebras is widely acknowledged for finite-dimensional evolution algebras. For a review of relevant applications, we refer the reader to \cite{Cabrera/Siles/Velasco,Elduque/Labra/2015,Elduque/Labra/2016}. 

\subsection{The associated digraphs after a change of basis}
It is worth pointing out that the obtained graph $G(\A,\mathcal{B})$ depends on the chosen basis as the following examples show.

\bigskip
\begin{example}\label{exa:base1}
Let $\A$ be the Hilbert evolution algebra with  natural orthonormal basis $\mathcal{B}=\{e_i\}_{i\in \mathbb{N}}$, and multiplication given by:
$$
e_i^2=\left\{
\begin{array}{cl}
 e_{i+1} + e_{i+2},    &\text{ for }i\in \mathbb{N}\text{ odd},\\[.2cm]
 e_{i} + e_{i+1},    &\text{ for }i\in \mathbb{N}\text{ even},
\end{array}
\right.
$$ 
and $e_i \cdot e_j= 0, \text{ for }i\neq j.$ See the associated weighted digraph $G(\A,\mathcal{B})$ in Figure \ref{fig:base1}(a). Let $\mathcal{C}:=\{f_i\}_{i\in \mathbb{N}}$ be given by:
\begin{equation}
\label{eq:ei to fi}
f_i:=\left\{
\begin{array}{cl}
 e_{i} + e_{i+1},    &\text{ for }i\in \mathbb{N}\text{ odd},\\[.2cm]
 e_{i} - e_{i-1},    &\text{ for }i\in \mathbb{N}\text{ even},
\end{array}
\right.
\end{equation} 
and let $\mathcal{C}_{0}:=\{\tilde{f}_{i}\}_{i\in \mathbb{N}}$ with $\tilde{f}_i:=f_i/\|f_i\|$ for $i\in\mathbb{N}$.  We claim that $\mathcal{C}_0$ is another natural orthonormal basis of $\A$. Indeed, thus defined, we can write
\begin{equation}
\label{eq:fi to ei}
e_i =\left\{
\begin{array}{cl}
(1/2) f_{i} - (1/2)f_{i+1},    &\text{ for }i\in \mathbb{N}\text{ odd},\\[.2cm]
 (1/2) f_{i} + (1/2) f_{i-1},    &\text{ for }i\in \mathbb{N}\text{ even}.
\end{array}
\right.
\end{equation} 
\noindent
In order to verify that it is natural, observe that $f_{i}\cdot f_{j}=0$ if $|j-i|\geq 2$, and if $j=i+1$ then we have for $i\in \mathbb{N}$ even that $f_i \cdot f_j:= (e_{i} - e_{i-1})\cdot (e_{i+1} + e_{i+2}) = 0,$ while  for $i\in \mathbb{N}$ odd we obtain:
$$f_i \cdot f_j:= (e_{i} + e_{i+1})\cdot (e_{i+1} - e_{i}) = e_{i+1}^2 -e_i^2 = (e_{i+1}+e_{i+2})- (e_{i+1}+e_{i+2})=0.$$
In a similar way  we can verify that $\langle f_i , f_j \rangle =0$, for $i\not= j$. So $\mathcal{C}_0$ is a natural orthonormal set of $\A$.
But the equations \eqref{eq:ei to fi} and \eqref{eq:fi to ei} are bijections, thus the spanned sets by $\mathcal{B}$ and $\mathcal{C}_0$ are equal, implying that $\mathcal{C}_0$ is a natural orthonormal basis of $\A$.
Moreover, the multiplication of $\A$ using $\mathcal{C}_0$ is given by:
\begin{equation}\label{eq:foeq1}
\tilde{f}_i^2=\left\{
\begin{array}{cl}
 (1/\sqrt{2})\left\{\tilde{f}_i + \tilde{f}_{i+1} + \tilde{f}_{i+2} -  \tilde{f}_{i+3}\right\},    &\text{ for }i\in \mathbb{N}\text{ odd},\\[.2cm]
(1/\sqrt{2})\left\{\tilde{f}_{i-1} +  \tilde{f}_{i} +  \tilde{f}_{i+1} -  \tilde{f}_{i+2}\right\},    &\text{ for }i\in \mathbb{N}\text{ even},
\end{array}
\right.
\end{equation} 
and $\tilde{f}_i \cdot \tilde{f}_j= 0, \text{ for }i\neq j.$ Let us check \eqref{eq:foeq1}. Note that $\tilde{f}_{i}^2 = f_i^2 /2$. Let $i\in \mathbb{N}$ odd. Then
$$
\begin{array}{ccl}
\tilde{f}_{i}^2 &=&(1/2)e_i^2 + (1/2)e_{i+1}^2\\[.2cm]
 &=&(1/2)(e_{i+1} + e_{i+2}) + (1/2)(e_{i+1} + e_{i+2})\\[.2cm]
 &=&e_{i+1} + e_{i+2}\\[.2cm]
  &=&\left\{(1/2)f_{i+1} + (1/2)f_{i}\right\} +\left\{ (1/2)f_{i+2} - (1/2)f_{i+3}\right\}.\\[.2cm]
  &=& (1/\sqrt{2})\left\{\tilde{f}_i + \tilde{f}_{i+1} + \tilde{f}_{i+2} -  \tilde{f}_{i+3}\right\}.
\end{array}
$$
\smallskip
Analogously, if $i\in \mathbb{N}$ is even, we get
$$
\begin{array}{ccl}
\tilde{f}_{i}^2 &=&(1/2)e_i^2 + (1/2)e_{i-1}^2\\[.2cm]
 &=&e_{i} + e_{i+1}\\[.2cm]
  &=&\left\{(1/2)f_{i} + (1/2)f_{i-1}\right\} +\left\{ (1/2)f_{i+1} - (1/2)f_{i+2}\right\}\\[.2cm]
  &=& (1/\sqrt{2})\left\{\tilde{f}_{i-1} +\tilde{f}_i + \tilde{f}_{i+1} - \tilde{f}_{i+2} \right\}.
\end{array}
$$

\smallskip
The associated weighted digraph $G(\A,\mathcal{C}_{0})$ is given in Figure \ref{fig:base1}(b).

\end{example}

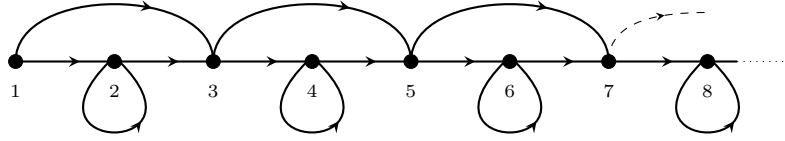
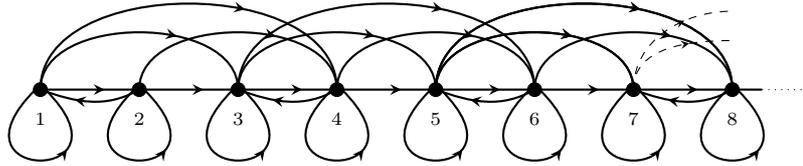
\begin{figure}[h!]
    \centering
    \subfigure[$G(\A,\mathcal{B})$]{
    \begin{tikzpicture}[scale=1.3]
\draw[thick] (7,0) -- (7.3,0);
\draw[thick,directed] (0,0) -- (1,0);
\draw [thick,directed] (0,0) to [out=90,in=90] (2,0);
\draw [thick, directed] (0.95,0) to [out=225,in=315,looseness=35] (1.05,0);
\draw[thick,directed] (1,0) -- (2,0);
\draw [thick,directed] (2,0) to [out=90,in=90] (4,0);
\draw[thick,directed] (2,0) -- (3,0);
\draw [thick, directed] (2.95,0) to [out=225,in=315,looseness=35] (3.05,0);
\draw[thick,directed] (3,0) -- (4,0);
\draw [thick,directed] (4,0) to [out=90,in=90] (6,0);
\draw[thick,directed] (4,0) -- (5,0);
\draw [thick, directed] (4.95,0) to [out=225,in=315,looseness=35] (5.05,0);
\draw[thick,directed] (5,0) -- (6,0);
\draw [dashed,directed] (6,0) to [out=90,in=180] (7,0.5);
\draw[thick,directed] (6,0) -- (7,0);
\draw [thick, directed] (6.95,0) to [out=225,in=315,looseness=35] (7.05,0);
\draw[dotted] (7.3,0) -- (7.8,0);
\filldraw [black] (0,0) circle (2pt);    \node at (0,-0.3) {\footnotesize $1$};     
\filldraw [black] (1,0) circle (2pt);     \node at (1,-0.3) {\footnotesize $2$}; 
\filldraw [black] (2,0) circle (2pt);     \node at (2,-0.3) {\footnotesize $3$}; 
\filldraw [black] (3,0) circle (2pt);     \node at (3,-0.3) {\footnotesize $4$}; 
\filldraw [black] (4,0) circle (2pt);     \node at (4,-0.3) {\footnotesize $5$}; 
\filldraw [black] (5,0) circle (2pt);     \node at (5,-0.3) {\footnotesize $6$}; 
\filldraw [black] (6,0) circle (2pt);     \node at (6,-0.3) {\footnotesize $7$}; 
\filldraw [black] (7,0) circle (2pt);     \node at (7,-0.3) {\footnotesize $8$}; 
    \end{tikzpicture}}

        \subfigure[$G(\A,\mathcal{C}_0)$.]{
    \begin{tikzpicture}[scale=1.3]
\draw[thick] (7,0) -- (7.3,0);
\draw[thick,directed] (0,0) -- (1,0);
\draw [thick, directed] (-0.05,0) to [out=225,in=315,looseness=35] (0.05,0);
\draw [thick,directed] (0,0) to [out=90,in=90] (2,0);
\draw [thick,directed] (0,0) to [out=90,in=90] (3,0);
\draw [thick, directed] (0.95,0) to [out=225,in=315,looseness=35] (1.05,0);
\draw[thick,directed] (1,0) -- (2,0);
\draw [thick,directed] (1,0) to [out=90,in=90] (3,0);
\draw [thick,directed] (1,0) to [out=210,in=340] (0,0);
\draw [thick,directed] (2,0) to [out=90,in=90] (4,0);
\draw [thick,directed] (2,0) to [out=90,in=90] (5,0);
\draw [thick, directed] (1.95,0) to [out=225,in=315,looseness=35] (2.05,0);
\draw[thick,directed] (2,0) -- (3,0);
\draw [thick, directed] (2.95,0) to [out=225,in=315,looseness=35] (3.05,0);
\draw[thick,directed] (3,0) -- (4,0);
\draw [thick,directed] (3,0) to [out=90,in=90] (5,0);
\draw [thick,directed] (3,0) to [out=210,in=340] (2,0);
\draw [thick,directed] (4,0) to [out=90,in=90] (6,0);
\draw [thick,directed] (4,0) to [out=90,in=90] (7,0);
\draw [thick, directed] (3.95,0) to [out=225,in=315,looseness=35] (4.05,0);
\draw[thick,directed] (4,0) -- (5,0);
\draw [thick,directed] (4,0) to [out=90,in=90] (6,0);
\draw [thick,directed] (4,0) to [out=90,in=90] (7,0);
\draw [thick, directed] (4.95,0) to [out=225,in=315,looseness=35] (5.05,0);
\draw[thick,directed] (5,0) -- (6,0);
\draw [thick,directed] (5,0) to [out=90,in=90] (7,0);
\draw [thick,directed] (5,0) to [out=210,in=340] (4,0);
\draw [dashed,directed] (6,0) to [out=90,in=180] (7,0.5);
\draw [dashed,directed] (6,0) to [out=90,in=180] (7,0.8);
\draw [thick, directed] (5.95,0) to [out=225,in=315,looseness=35] (6.05,0);
\draw[thick,directed] (6,0) -- (7,0);
\draw [thick, directed] (6.95,0) to [out=225,in=315,looseness=35] (7.05,0);
\draw [thick,directed] (7,0) to [out=210,in=340] (6,0);
\draw[dotted] (7.3,0) -- (7.8,0);
\filldraw [black] (0,0) circle (2pt);    \node at (0,-0.3) {\footnotesize $1$};     
\filldraw [black] (1,0) circle (2pt);     \node at (1,-0.3) {\footnotesize $2$}; 
\filldraw [black] (2,0) circle (2pt);     \node at (2,-0.3) {\footnotesize $3$}; 
\filldraw [black] (3,0) circle (2pt);     \node at (3,-0.3) {\footnotesize $4$}; 
\filldraw [black] (4,0) circle (2pt);     \node at (4,-0.3) {\footnotesize $5$}; 
\filldraw [black] (5,0) circle (2pt);     \node at (5,-0.3) {\footnotesize $6$}; 
\filldraw [black] (6,0) circle (2pt);     \node at (6,-0.3) {\footnotesize $7$}; 
\filldraw [black] (7,0) circle (2pt);     \node at (7,-0.3) {\footnotesize $8$}; 
    \end{tikzpicture}}
    \caption{
    Associated digraphs for the algebra $\A$ from Example \ref{exa:base1}.}
\label{fig:base1}
\end{figure}

\bigskip

\begin{example}\label{exa:base2}
Let $\A$ be the Hilbert evolution algebra with  natural  orthonormal basis $\mathcal{B}=\{e_i\}_{i\in \mathbb{N}}$, and multiplication given by:
$$
e_i^2=\left\{
\begin{array}{cl}
\sum_{\ell=2}^{\infty} \alpha_{\ell}e_{\ell} ,    &\text{ for }i=1,  \\[.2cm]
 e_{i} + e_{i-1},    &\text{ for }i\in \mathbb{N}\setminus\{1\}\text{ odd},\\[.2cm]
 e_{i} + e_{i+1},    &\text{ for }i\in \mathbb{N}\setminus\{1\}\text{ even},
\end{array}
\right.
$$ 
and $e_i \cdot e_j= 0, \text{ for }i\neq j.$ The associated graph $G(\A,\mathcal{B})$ is given in Figure \ref{fig:base2}. Let $\mathcal{C}:=\{f_i\}_{i\in \mathbb{N}}$ be given by: $f_1:=e_1$, and 
$$f_i:=\left\{
\begin{array}{cl}
e_{i} + e_{i-1},    &\text{ for }i\in \mathbb{N}\setminus\{1\}\text{ odd},\\[.2cm]
 e_{i} - e_{i+1},    &\text{ for }i\in \mathbb{N}\setminus\{1\}\text{ even},
\end{array}
\right.
$$ 
and note that it is another natural basis of $\A$. In this case, we can write $e_1=f_1$ and
$$e_i =\left\{
\begin{array}{cl}
(1/2) f_{i} - (1/2)f_{i-1},    &\text{ for }i\in \mathbb{N}\text{ odd},\\[.2cm]
 (1/2) f_{i} + (1/2) f_{i+1},    &\text{ for }i\in \mathbb{N}\text{ even},
\end{array}
\right.
$$ 
so $\mathcal{C}$ is a basis of $\A$. In addition, $f_1 \cdot f_j =0$ for $j\neq 1$ and $f_{i}\cdot f_{j}=0$ if $|j-i|\geq 2$. On the other hand, if $j=i+1$ note that for $i\in \mathbb{N}\setminus\{1\}$ odd we have $f_i \cdot f_j:= (e_{i} + e_{i-1})\cdot (e_{i+1} - e_{i+2}) = 0,$ while, for $i\in \mathbb{N}\setminus\{1\}$  even, we get
$$
f_i \cdot f_j:= (e_{i} - e_{i+1})\cdot (e_{i+1} + e_{i}) = e_{i}^2 - e_{i+1}^2 = (e_i + e_{i+1}) - (e_{i+1}+ e_i) =0.
$$ 
In a similar way to Example \ref{exa:base1}, we can prove that $\mathcal{C}$ is also an orthogonal basis.Therefore, $\mathcal{C}_{0}:=\{\tilde{f}_{i}\}_{i\in \mathbb{N}}$ with $\tilde{f}_i:=f_i/\|f_i\|$ for $i\in\mathbb{N}$ is an orthonormal basis of $\A$ with multiplication given by:
\begin{equation}\nonumber
    \tilde{f}_1^2 = (1/\sqrt{2})\left\{ \displaystyle\sum_{\ell=1}^{\infty} \left(\alpha_{2\ell} -\alpha_{2\ell +1}\right) \tilde{f}_{2\ell} +  \sum_{\ell=1}^{\infty} \left(\alpha_{2\ell} +\alpha_{2\ell +1}\right) \tilde{f}_{2\ell +1}\right\},
\end{equation}

\begin{equation}\nonumber
\label{eq:foeq2}
\tilde{f}_i^2=\left\{
\begin{array}{cl}  
 \sqrt{2} \tilde{f}_i ,    &\text{ for }i\in \mathbb{N}\setminus\{1\}\text{ odd},\\[.2cm]
\sqrt{2} \tilde{f}_{i+1},    &\text{ for }i\in \mathbb{N}\text{ even},
\end{array}
\right.
\end{equation} 
and $\tilde{f}_i \cdot \tilde{f}_j= 0, \text{ for }i\neq j.$ It is worth pointing out that the graph associated to $\A$ through the basis $\mathcal{C}_0$ depends of the choose of the sequence $\{\alpha_\ell\}_{\ell\in \mathbb{N}}$. See two of these chooses in Figure \ref{fig:base2}.

\end{example}

\begin{figure}[h!]
    \centering
    \subfigure[$G(\A,\mathcal{B})$.]{
\begin{tikzpicture}[scale=1.2]

\draw[directed] (0,0) -- (-1.8,1.1);
\draw [directed] (-1.8,1.05) to [out=215,in=145,looseness=25] (-1.8,1.15);
\draw[directed] (-1.8,1.1)--(-1.1,1.8);
\draw [directed] (-1.15,1.75) to [out=160,in=90,looseness=17] (-1.05,1.85);
\draw [directed] (-1.1,1.8) to [out=270,in=10] (-1.8,1.1);
\draw[directed] (0,0)--(-1.1,1.8);

\draw[directed] (0,0) -- (-0.8,2);
\draw[directed] (0.8,2) to [out=215,in=325] (-0.8,2);
\draw [directed] (-0.85,2) to [out=125,in=55,looseness=25] (-0.75,2);
\draw[directed] (-0.8,2)--(0.8,2);
\draw [directed] (0.75,2) to [out=125,in=55,looseness=25] (0.85,2);
\draw[directed] (0,0) --(0.8,2);

\draw[directed] (0,0) -- (1.8,1.1);
\draw[directed] (1.1,1.8)--(1.8,1.1);
\draw [directed] (1.1,1.85) to [out=90,in=20,looseness=25] (1.1,1.75);
\draw[directed] (1.8,1.1) to [out=190,in=270] (1.1,1.8);
\draw [directed] (1.8,1.15) to [out=70,in=360,looseness=25] (1.8,1.05);
\draw[directed] (0,0)--(1.1,1.8);

\draw[directed] (0,0) -- (2,0.8);
\draw[directed] (2,-0.8) to [out=125,in=215] (2,0.8);
\draw [directed] (2,0.85) to [out=35,in=325,looseness=25] (2,0.75);
\draw[directed] (2,0.8)--(2,-0.8);
\draw [directed] (2,-0.75) to [out=35,in=325,looseness=25] (2,-0.85);
\draw[directed] (0,0)--(2,-0.8);


\draw[directed] (0,0) -- (1.8,-1.1);
\draw[directed] (1.8,-1.1)--(1.1,-1.8);
\draw [directed] (1.85,-1.1) to [out=360,in=270,looseness=35] (1.8,-1.15);
\draw[directed] (1.1,-1.8) to [out=90,in=190] (1.8,-1.1);
\draw[directed] (0,0)--(1.1,-1.8);
\draw [directed] (1.15,-1.8) to [out=305,in=235,looseness=25] (1.05,-1.8);


\filldraw [black] (0,0) circle (2.1pt);
\draw (-1.4,1.55) node[above,font=\footnotesize] {$3$};
\filldraw [black] (2,0.8) circle (2.1pt);
\draw (1.4,1.55) node[above,font=\footnotesize] {$6$};
\filldraw [black] (2,-0.8) circle (2.1pt);
\draw (0.8,-1.55) node[below,font=\footnotesize] {$11$};
\filldraw [black] (-0.8,2) circle (2.1pt);
\draw (-0.5,2) node[above,font=\footnotesize] {$4$};
\filldraw [black] (0.8,2) circle (2.1pt);
\draw (0.5,2) node[above,font=\footnotesize] {$5$};

\filldraw [black] (1.1,-1.8) circle (2.1pt);
\filldraw [black] (1.1,1.8) circle (2.1pt);
\filldraw [black] (-1.1,1.8) circle (2.1pt);

\filldraw [black] (-1.8,1.1) circle (2.1pt);
\draw (-1.8,0.8) node[font=\footnotesize] {$2$};
\filldraw [black] (1.8,1.1) circle (2.1pt);
\draw (1.75,1.2) node[above,font=\footnotesize] {$7$};
\filldraw [black] (1.8,-1.1) circle (2.1pt);
\draw (2.15,0.6) node[below,font=\footnotesize] {$8$};
\draw (2.15,-0.6) node[above,font=\footnotesize] {$9$};
\draw (1.7,-1.3) node[below,font=\footnotesize] {$10$};

\draw (-0.2,-0.2) node[font=\footnotesize] {$1$};

\draw [thick,loosely dotted]    (-1.2,0.2) to[out=240,in=-145] (0.2,-1.2);

\end{tikzpicture}}

    \subfigure[$G(\A,\mathcal{C}_{0})$ obtained by assuming $\alpha_{2\ell}\neq \pm\, \alpha_{2\ell+1}$ for all $\ell\in \mathbb{N}$.]{
\begin{tikzpicture}[scale=1.2]

\draw[directed] (0,0) -- (-1.8,1.1);
\draw[directed] (-1.8,1.1)--(-1.1,1.8);
\draw [directed] (-1.15,1.75) to [out=160,in=90,looseness=17] (-1.05,1.85);
\draw[directed] (0,0)--(-1.1,1.8);

\draw[directed] (0,0) -- (-0.8,2);
\draw[directed] (-0.8,2)--(0.8,2);
\draw [directed] (0.75,2) to [out=125,in=55,looseness=25] (0.85,2);
\draw[directed] (0,0) --(0.8,2);

\draw[directed] (0,0) -- (1.8,1.1);
\draw[directed] (1.1,1.8)--(1.8,1.1);
\draw [directed] (1.8,1.15) to [out=70,in=360,looseness=25] (1.8,1.05);
\draw[directed] (0,0)--(1.1,1.8);

\draw[directed] (0,0) -- (2,0.8);
\draw[directed] (2,0.8)--(2,-0.8);
\draw [directed] (2,-0.75) to [out=35,in=325,looseness=25] (2,-0.85);
\draw[directed] (0,0)--(2,-0.8);


\draw[directed] (0,0) -- (1.8,-1.1);
\draw[directed] (1.8,-1.1)--(1.1,-1.8);
\draw[directed] (0,0)--(1.1,-1.8);
\draw [directed] (1.15,-1.8) to [out=305,in=235,looseness=25] (1.05,-1.8);


\filldraw [black] (0,0) circle (2.1pt);
\draw (-1.4,1.55) node[above,font=\footnotesize] {$3$};
\filldraw [black] (2,0.8) circle (2.1pt);
\draw (1.4,1.55) node[above,font=\footnotesize] {$6$};
\filldraw [black] (2,-0.8) circle (2.1pt);
\draw (0.8,-1.55) node[below,font=\footnotesize] {$11$};
\filldraw [black] (-0.8,2) circle (2.1pt);
\draw (-0.5,2) node[above,font=\footnotesize] {$4$};
\filldraw [black] (0.8,2) circle (2.1pt);
\draw (0.5,2) node[above,font=\footnotesize] {$5$};

\filldraw [black] (1.1,-1.8) circle (2.1pt);
\filldraw [black] (1.1,1.8) circle (2.1pt);
\filldraw [black] (-1.1,1.8) circle (2.1pt);

\filldraw [black] (-1.8,1.1) circle (2.1pt);
\draw (-1.8,0.8) node[font=\footnotesize] {$2$};
\filldraw [black] (1.8,1.1) circle (2.1pt);
\draw (1.75,1.2) node[above,font=\footnotesize] {$7$};
\filldraw [black] (1.8,-1.1) circle (2.1pt);
\draw (2.15,0.6) node[below,font=\footnotesize] {$8$};
\draw (2.15,-0.6) node[above,font=\footnotesize] {$9$};
\draw (1.7,-1.3) node[below,font=\footnotesize] {$10$};

\draw (-0.2,-0.2) node[font=\footnotesize] {$1$};

\draw [thick,loosely dotted]    (-1.2,0.2) to[out=240,in=-145] (0.2,-1.2);

\end{tikzpicture}}\qquad\subfigure[$G(\A,\mathcal{C}_{0})$ obtained by assuming $\alpha_{2\ell}= \alpha_{2\ell+1}$ for all $\ell \in \mathbb{N}$.]{
\begin{tikzpicture}[scale=1.2]

\draw[directed] (-1.8,1.1)--(-1.1,1.8);
\draw [directed] (-1.15,1.75) to [out=160,in=90,looseness=17] (-1.05,1.85);
\draw[directed] (0,0)--(-1.1,1.8);

\draw[directed] (-0.8,2)--(0.8,2);
\draw [directed] (0.75,2) to [out=125,in=55,looseness=25] (0.85,2);
\draw[directed] (0,0) --(0.8,2);

\draw[directed] (0,0) -- (1.8,1.1);
\draw[directed] (1.1,1.8)--(1.8,1.1);
\draw [directed] (1.8,1.15) to [out=70,in=360,looseness=25] (1.8,1.05);

\draw[directed] (2,0.8)--(2,-0.8);
\draw [directed] (2,-0.75) to [out=35,in=325,looseness=25] (2,-0.85);
\draw[directed] (0,0)--(2,-0.8);


\draw[directed] (1.8,-1.1)--(1.1,-1.8);
\draw[directed] (0,0)--(1.1,-1.8);
\draw [directed] (1.15,-1.8) to [out=305,in=235,looseness=25] (1.05,-1.8);


\filldraw [black] (0,0) circle (2.1pt);
\draw (-1.4,1.55) node[above,font=\footnotesize] {$3$};
\filldraw [black] (2,0.8) circle (2.1pt);
\draw (1.4,1.55) node[above,font=\footnotesize] {$6$};
\filldraw [black] (2,-0.8) circle (2.1pt);
\draw (0.8,-1.55) node[below,font=\footnotesize] {$11$};
\filldraw [black] (-0.8,2) circle (2.1pt);
\draw (-0.5,2) node[above,font=\footnotesize] {$4$};
\filldraw [black] (0.8,2) circle (2.1pt);
\draw (0.5,2) node[above,font=\footnotesize] {$5$};

\filldraw [black] (1.1,-1.8) circle (2.1pt);
\filldraw [black] (1.1,1.8) circle (2.1pt);
\filldraw [black] (-1.1,1.8) circle (2.1pt);

\filldraw [black] (-1.8,1.1) circle (2.1pt);
\draw (-1.8,0.8) node[font=\footnotesize] {$2$};
\filldraw [black] (1.8,1.1) circle (2.1pt);
\draw (1.75,1.2) node[above,font=\footnotesize] {$7$};
\filldraw [black] (1.8,-1.1) circle (2.1pt);
\draw (2.15,0.6) node[below,font=\footnotesize] {$8$};
\draw (2.15,-0.6) node[above,font=\footnotesize] {$9$};
\draw (1.7,-1.3) node[below,font=\footnotesize] {$10$};

\draw (-0.2,-0.2) node[font=\footnotesize] {$1$};

\draw [thick,loosely dotted]    (-1.8,-0.2) to[out=270,in=180] (-0.2,-1.8);
\end{tikzpicture}}
    \caption{Associated digraphs of the algebra $\A$ of Example \ref{exa:base2}. }
    \label{fig:base2}
\end{figure}
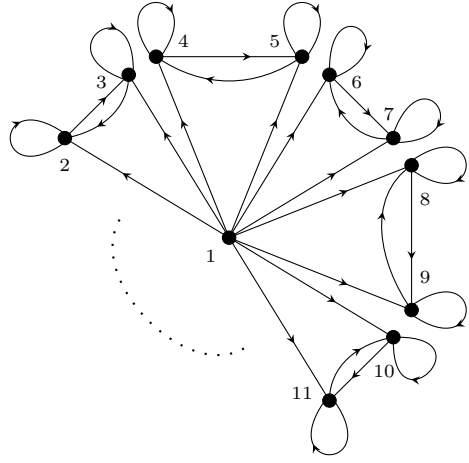
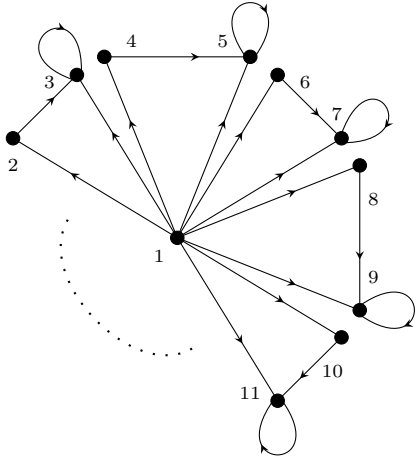
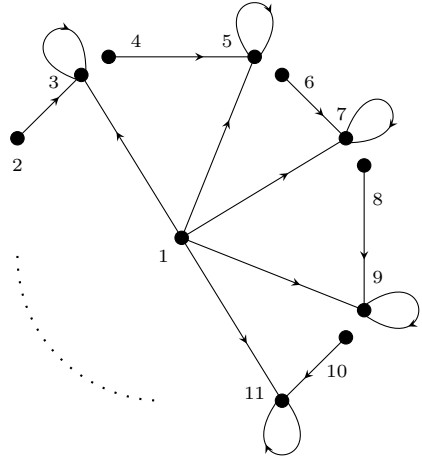


In both, Example \ref{exa:base1} and Example \ref{exa:base2}, the new basis $\mathcal{C}$ obtained from $\mathcal{B}$ is an orthogonal basis so the corresponding natural orthonormal basis $\mathcal{C}_0$ arises trivially. In general, a natural question when studying (Hilbert) evolution algebras is about the behaviour of the product of the algebra under the change of basis.

\section{Nilpotency}\label{nilpotency}
Inspired by the analysis of \cite{Elduque/Labra/2015}, we discuss how a condition for nilpotency can be affected whether we consider Hilbert evolution algebras.

Let $\A$ be a commutative algebra. The principal powers of $v \in \A$ are the products $v^2:= v\cdot v$, and in general $v^{n+1}:= v^n \cdot v$,  for $n\in \mathbb{N}$. We say that $v\in \A$ is nil if there exists a natural number $n$ such that $v^n= 0$. We say that $\A$ is nil if for any $v\in \A$, $v$ is nil. As usual, see \cite{camacho/gomez/omirov/turdibaev/2013,Elduque/Labra/2015},  we denote the following sequences of subspaces
$$
\begin{array}{lll}
\A^{<1>}:=\A,& &\A^{<n+1>}:=\A^{<n>}\A;\\[.2cm]
\A^{1}:=\A,& &\A^{n+1}: = \sum_{i=1}^{n}\A^i \A^{n+1-i},
\end{array}
$$
for any $n\in\mathbb{N}$. Moreover,  we say that $\A$ is (right) nilpotent if there exists $n\in\mathbb{N}$ such that ($\A^{<n>} = 0$) $\A^n =0$, and the minimal of such number is called the index of (right) nilpotency. That is, the index of right nilpotency is given by $n_r(\A):=\min\{n\in\mathbb{N}:\A^{<n>}=0\}$. For  a commutative algebra $\A$ we have that $\A^{2^n}\subset\A^{<n>}$ for $n\geq 1$ (see \cite[Prop 1, Chap 4]{Zhevlakov}) and we know that $\A^{<n>}\subset \A^n$, thus right nilpotency is equivalent to nilpotency. These results can be applied to Hilbert evolution algebras.
Besides this, we are going to introduce the following useful notation.
Let $I\subset \A$ be a subset of the algebra, define $I^{<1>}:=I$, and for any $n\in\mathbb{N}$, let $I^{<n+1>}:=I^{<n>}I$. That is,
$$
I^{<n+1>}:=\left\{\sum_{i}x_i y_i : x_i\in I^{<n>}, y_i \in I\right\}.
$$

We point out that for any finite-dimensional evolution algebra it is possible to define a norm such that the algebra becomes a Hilbert evolution algebra. Before analyzing (right) nilpotency for Hilbert evolution algebras let us remember what is known for evolution algebras.

\bigskip
\begin{theorem}\label{theo:duque}
\cite[Theorem 3.4]{Elduque/Labra/2015} Let $\A$ be a finite-dimensional evolution algebra with a natural basis $\mathcal{B}$. Then, the following conditions are equivalent:
\begin{enumerate}[label=(\roman*)]
    \item $\A$ is nil.
\item There are no oriented cycles in $G(\A, \mathcal{B})$.
\item The basis $\mathcal{B}$ can be reordered so that the matrix of structural constants is strictly upper triangular.
\item $\A$ is nilpotent.
\end{enumerate}
\end{theorem}

The previous theorem gains in interest if we realize that nilpotency of the algebra can be studied from the associated digraph. We shall see that it is also true for the infinite-dimensional case. However, some equivalences of the theorem are not satisfied when we consider arbitrary Hilbert evolution algebras, as evidenced by the following results.

 Let $\A$  a Hilbert evolution algebra with an orthonormal natural basis $\mathcal{B}= \{e_i\}_{i\in \mathbb{N}}$. Let us note that,  by the definition of the $G(\A, \mathcal{B})$,  $D^1(i)= \{k\in \mathbb{N}: w_{ik} \not= 0\},$
and consequently, we can write $e^2_{i}= \sum_{k\in D^1(i)}w_{ik}e_k,$ for all $i\in\mathbb{N}.$ Moreover, if $U\subset\mathbb{N}$ then
$$D^m(U)= \{k \in \mathbb{N}: w_{ik} \not= 0 \text{ and } i\in D^{m-1}(U)\},$$
and we can write 
\begin{equation}
\label{eq:span of ei^2}
e^2_{i}= \sum_{k\in D^m(i)}w_{ik}e_k,
\end{equation}
for all $i\in D^{m-1}(U).$


\bigskip
\begin{proposition}\label{prop:seila}
Let $\A$ be a Hilbert evolution algebra with a  natural  orthonormal basis $\{e_i\}_{i\in\mathbb{N}}$. If $U\subset\mathbb{N}$ and $I:=\overline{\span}\{e_k:k\in U\}$ then
\begin{equation}\nonumber 
\label{eq:lema1 1}
I^{<n>} \subset\overline{\span}\{e_k:k \in D^{n-1}(U)\} \text{ for } n\in\mathbb{N}.
\end{equation}
\end{proposition}
\begin{proof} 
The proof is by induction on $n$. The case $n=1$ is direct since $I^{<1>} = I = \overline{\span}\{e_k:k \in U\}.$ Assume that the statement holds for any $k\leq n$. Let us prove it for $n+1$. Take $v\in I^{<n+1>}$, then $v=u\cdot w$ for $u\in I^{<n>}$ and $w\in I$. Hence, 
we have $$u\cdot w =\left(\sum_{k\in D^{n-1}(U)} \alpha_k e_k\right) \left(\sum_{j=1}^\infty  \beta_j e_j\right)=\sum_{k\in D^{n-1}(U)} \alpha_k\beta_k e_k^2.$$
That is
$$I^{<n+1>}=I^{<n>}I \subset \overline{\span}\{e_k^2:k \in D^{n-1}(U)\}\subset\overline{\span}\{e_k:k \in D^{n}(U)\},$$
where we use Equation \eqref{eq:span of ei^2}.




\end{proof}

\bigskip
\begin{corollary}\label{lem:Angraph}
    Let $\A$ be a Hilbert evolution algebra with a natural orthonormal basis $\{e_i\}_{i\in\mathbb{N}}$. Then, for any $n\in\mathbb{N}$
\begin{equation}\nonumber
    \A^{<n>} \subset \overline{\span}\{e_k:k \in D^{n-1}(\mathbb{N})\}.
\end{equation}
\end{corollary}

\begin{proof} If we put  $U= \mathbb{N}$ in the statement of Proposition \ref{prop:seila}, then  $I=\A$ and our assertion follows.
\end{proof}

Corollary \ref{lem:Angraph} provides an intuitive criterion, from $G(\A,\mathcal{B})$, for the right nilpotency of a Hilbert evolution algebra. Indeed, note that if $G(\A,\mathcal{B})$ is such that there exists $n\in\mathbb{N}$ such that $D^{n-1}(\mathbb{N})=\emptyset$ then $\A^{<n>}=0$. This intuition will be useful to analyse the conditions of Theorem \ref{theo:duque} in the context of Hilbert evolution algebras. Let us start with some examples.


\bigskip
\begin{example}
Examples where $G(\A,\mathcal{B})$ has no oriented cycles but $\A$ is not nilpotent. 
\begin{enumerate}
    \item Consider the algebra $\A$ from Example \ref{ex:tree-HEA}. In Figure \ref{fig:tree-HEA} we illustrate the associated digraph $G(\A, \mathcal{B})$, which is a rooted $r$-ary tree. For the sake of simplicity take $r=2$, and let $v=\sum_{k\in \Lambda} \alpha_k e_k \in \A$, with $\alpha_k \neq 0$ for any $k\in\Lambda$. It is not difficult to see that $0\neq v^n\in A^{<n>}$ for any $n\in \mathbb{N}$. This can be proved by induction by noting that
    $$v^2=v\cdot v=\sum_{k\in \Lambda}\alpha_k^2 e_k^2=\sum_{k\in \Lambda\setminus \{1\}} \alpha_k^2(c_{k,k1}e_{k1}+c_{k,k2}e_{k2}).$$
    
    \smallskip
    \item Consider the Hilbert evolution algebra from Example \ref{exa:nonLF}. In Figure \ref{fig:nonLF} we illustrate the associated digraph $G(\A, \mathcal{B})$, which is a non-locally finite graph with infinitely many vertices. In this case, take  $v=\sum_{k\in \mathbb{N}\setminus \{1\}} \alpha_k e_k \in \A$, with $\alpha_k \neq 0$ for $k\in\mathbb{N}\setminus \{1\}$. Note that $0\neq v^n \in \A^{n}$ for any $n\in\mathbb{N}$. Indeed,
    $$v^2=v\cdot v=\sum_{k\in\mathbb{N}\setminus \{1\}} \alpha_k^2 e_k^2= \sum_{k\in\mathbb{N}\setminus \{1,2\}} \alpha_{k-1}^2 e_k,$$
        $$v^3=v^2\cdot v=\left( \sum_{k\in\mathbb{N}\setminus \{1,2\}} \alpha_{k-1}^2 e_k\right) \left(\sum_{k\in\mathbb{N}\setminus \{1\}} \alpha_k e_k\right)=\sum_{k\in\mathbb{N}\setminus \{1,2\}} \alpha_{k-1}^2 \alpha_k e_k^2,$$
        so
        $$v^3=\sum_{k\in\mathbb{N}\setminus \{1,2,3\}} \alpha_{k-2}^2 \alpha_{k-1} e_k,$$
        and, in general
        $$v^n = \sum_{k\in\mathbb{N}\setminus [n]} \alpha_{k-n+1}^2 \left\{\prod_{i=1}^{k-n+2} \alpha_{i} \right\}e_k ,$$
where $[n]:=\{1,\ldots,n\}$.
    
\end{enumerate}
\end{example}

For the following results we recall some definitions.
Given an associated weighted digraph $G(\A, \mathcal{B})=(V,E)$, for any $i\in V$,
we define $G_i(\A, \mathcal{B}):=(V_i, E_i)$, with $V_i:=D(i)$ and $E_i=\{(k, \ell): k, \ell \in V_i  \text{ and } (k, \ell)\in E\}$ and $\delta(G_i(\A, \mathcal{B})):=\sup\{\mathrm{dist}(i,u):u\in D(i)\}$ which is the depth of $G_i(\A, \mathcal{B})$.

\bigskip
\begin{theorem}\label{theo:duque-inf}
Let $\A$ be a Hilbert evolution algebra with an orthonormal natural basis $\mathcal{B}=\{e_i\}_{i\in\mathbb{N}}$ and associated digraph $G(\A, \mathcal{B})=(\mathbb{N}, E)$. Then, $\A$ is nilpotent if, and only if, the following conditions are satisfied:
\begin{enumerate}[label=(\roman*)]
\item \label{theo:duque-inf-1} there are no oriented cycles in $G(\A, \mathcal{B})$; 
\item  \label{theo:duque-inf-2} 
$\sup_{i\in \mathbb{N}} \delta\left(G_i(\A, \mathcal{B})\right)<\infty$.
\end{enumerate}
\end{theorem}

\begin{proof} 
Let $\A$ be a nilpotent algebra. The proof is by contradiction. The same arguments from the proof of \cite[Theorem 3.4]{Elduque/Labra/2015} apply to show that if $G(\A, \mathcal{B})$ has oriented cycles then $\A$ is not nilpotent. Indeed, let us assume that there exist an oriented cycle in $G(\A, \mathcal{B})$. Without loss of generality, assume that an oriented cycle of minimum length is the path $\{1,2,\ldots,k,1\}$.  By taking $v=\sum_{i=1}^k e_i$, the arguments of the aforementioned reference state that we have $0\neq v^n \in \A^n $, for any $n\in \mathbb{N}$. 
Now suppose that $\sup_{i\in \mathbb{N}} \delta\left(G_i(\A, \mathcal{B})\right)=\infty.$ If $\delta\left(G_i(\A, \mathcal{B})\right)<\infty$ for any $i$, there exist a sequences of positive integers $(i_k)_{k\in\mathbb{N}}$, such that $\delta\left(G_{i_k}(\A,\mathcal{B})\right)=n_{i_k},$ with $n_{i_1}<n_{i_2}<\cdots <n_{i_k} < \cdots$. Then, by considering $v=\sum_{k\in\mathbb{N}} \alpha_k e_{i_k}\in \A$, where the ${\alpha_k}$'s are not all zero, we conclude that $0\neq v^n \in \A^{<n>}$ for any $n\in\mathbb{N}$. On the other hand, if $\delta\left(G_i(\A, \mathcal{B})\right)=\infty$ for some $i\in\mathbb{N}$, then there exists a ray $\{v_1,v_2,\ldots \}$ contained in $G_i(\A, \mathcal{B})$. In this case, it is enough to take $u=\sum_{k\in\mathbb{N}} \beta_k e_{v_k}\in \A$ with $\beta_k\neq 0$, to conclude that $0\neq u^n \in \A^{<n>}$ for any $n\in\mathbb{N}$.



For the reciprocal, suppose that $\A$ is such that $G(\A, \mathcal{B})$ satisfies conditions \ref{theo:duque-inf-1} and \ref{theo:duque-inf-2}. By \ref{theo:duque-inf-2} there exists a constant $M<\infty$ such that $\delta\left(G_i(\A, \mathcal{B})\right)<M$ for all $i\in \mathbb{N}$. This, in turn, together with \ref{theo:duque-inf-1}, implies  $\A^{<n>}=0$ for any $n>M$. Indeed, by Corollary \ref{lem:Angraph} $\A^{<n>} \subset \overline{\span}\{e_k:k \in D^{n-1}(\mathbb{N})\},$ but \ref{theo:duque-inf-1} and \ref{theo:duque-inf-2} imply that  $\{e_k:k \in D^{n-1}(\mathbb{N})\}=\emptyset$ for any $n>M$. The proof is complete by noticing that $\A^{2^n}\subset \A^{<n>}$ for any $n\geq 1$.

\end{proof}

The results of Theorem \ref{theo:duque} (\cite[Theorem 3.4]{Elduque/Labra/2015}) are related to well-known results coming from Graph Theory for finite graphs. As far as we know, nilpotency on infinite locally-finite digraphs has been studied by \cite{anckutty}, see also \cite{ancykutty0}, so our approach may be useful to extend results in that direction as well. 

\bigskip
\begin{corollary}
Let $\A$ be a right nilpotent Hilbert evolution algebra with an orthonormal natural basis $\mathcal{B}=\{e_i\}_{i\in\mathbb{N}}$. Then, the index of right nilpotency of $\A$ is given by 
$$n_r(\A)=\sup_{i\in \mathbb{N}} \delta\left(G_i(\A, \mathcal{B})\right)+1.$$

\end{corollary}

\begin{proof}
Since $\A$ is right nilpotent, then $\A$ is nilpotent and by Theorem \ref{theo:duque-inf} 
$$\sup_{i\in \mathbb{N}} \delta\left(G_i(\A, \mathcal{B})\right)<\infty,    $$
for any $i\in \mathbb{N}$. Recall that the index of right nilpotency given by $n_r(\A):=\min\{n\in\mathbb{N}:\A^{<n>}=0\}$,
so since there are no oriented cycles in $G(\A, \mathcal{B})$, the proof is a consequence of Corollary \ref{lem:Angraph}.

\end{proof}


Another consequence of Theorem \ref{theo:duque} is the equivalence of nil and nilpotent for the finite case. However, although both conditions can be satisfied for infinite-dimensional Hilbert evolution algebras, this is not always true. 

\bigskip
\begin{example}\label{exa:nil-nilp}
($\A$ nil and nilpotent) Consider the Hilbert evolution algebra $\mathcal{A}$ whose resulting digraph $G(\A,\mathcal{B})$ is given by Figure \ref{fig:nil-nilpotent}. Let $v\in\A$ and note that a simple calculation shows that we always have $v^4=0$. Thus $\A$ is nil and nilpotent. Indeed, note that $G(\A, \mathcal{B})$ has no oriented cycles and that $\sup_{i\in \mathbb{N}} \delta\left(G_i(\A, \mathcal{B})\right)<\infty$ so nilpotency is also obtained by Theorem \ref{theo:duque-inf}. 
\end{example}

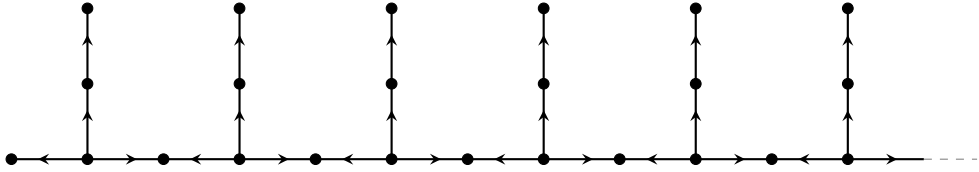
\begin{figure}[h!]
    \centering
    \begin{tikzpicture}

\draw[thick,directed] (-2,0) -- (-1,0);
\draw[thick,directed] (-2,0) -- (-3,0);
\draw[thick,directed] (0,0) -- (-1,0);
\draw[thick,directed] (0,0) -- (1,0);
\draw[thick,directed] (2,0) -- (1,0);
\draw[thick,directed] (2,0) -- (3,0);
\draw[thick,directed] (4,0) -- (3,0);
\draw[thick,directed] (4,0) -- (5,0);
\draw[thick,directed] (6,0) -- (5,0);
\draw[thick,directed] (6,0) -- (7,0);
\draw[thick,directed] (8,0) -- (7,0);
\draw[thick,directed] (8,0) -- (9,0);
\draw[dashed,gray] (9,0) -- (9.7,0);
\draw[thick,directed] (-2,0) -- (-2,1);
\draw[thick,directed] (-2,1) -- (-2,2);
\draw[thick,directed] (0,0) -- (0,1);
\draw[thick,directed] (0,1) -- (0,2);
\draw[thick,directed] (2,0) -- (2,1);
\draw[thick,directed] (2,1) -- (2,2);
\draw[thick,directed] (4,0) -- (4,1);
\draw[thick,directed] (4,1) -- (4,2);
\draw[thick,directed] (6,0) -- (6,1);
\draw[thick,directed] (6,1) -- (6,2);
\draw[thick,directed] (8,0) -- (8,1);
\draw[thick,directed] (8,1) -- (8,2);
\filldraw [black] (-2,2) circle (2pt); 
\filldraw [black] (0,2) circle (2pt); 
\filldraw [black] (-2,0) circle (2pt); 
\filldraw [black] (-3,0) circle (2pt); 
\filldraw [black] (-2,1) circle (2pt); 
\filldraw [black] (2,1) circle (2pt); 
\filldraw [black] (2,2) circle (2pt); 
\filldraw [black] (-1,0) circle (2pt); 
\filldraw [black] (0,0) circle (2pt);  
\filldraw [black] (1,0) circle (2pt);   
\filldraw [black] (2,0) circle (2pt);   
\filldraw [black] (3,0) circle (2pt);    
\filldraw [black] (4,0) circle (2pt);    
\filldraw [black] (5,0) circle (2pt);
\filldraw [black] (6,0) circle (2pt);    
\filldraw [black] (7,0) circle (2pt);    
\filldraw [black] (0,1) circle (2pt);   
\filldraw [black] (4,1) circle (2pt);
\filldraw [black] (4,2) circle (2pt);
\filldraw [black] (6,1) circle (2pt);
\filldraw [black] (6,2) circle (2pt);
\filldraw [black] (8,0) circle (2pt);
\filldraw [black] (8,1) circle (2pt);
\filldraw [black] (8,2) circle (2pt);
    \end{tikzpicture}
    \caption{Associated digraph for a Hilbert evolution algebra which is nil and nilpotent.}
    \label{fig:nil-nilpotent}
\end{figure}

\bigskip
\begin{example}
(Nil does not implies nilpotent) Consider the Hilbert evolution algebra $\mathcal{A}$ whose associated digraph $G(\A, \mathcal{B})$ is given by Figure \ref{fig:diam-go-inf}. Observe that, like Example \ref{exa:nil-nilp}, a simple calculation shows that we have $v^{k+2}=0$ for any $v\in G_{i_k}$. Thus $\A$ is nil. However, although there are no oriented cycles in $G(\A, \mathcal{B})$ nor rays, we have a sequence of positive integers $i_1,i_2,\ldots$ such that $\delta\left(G_{i_k}(\A, \mathcal{B})\right)=k$ so $\sup_{i\in \mathbb{N}} \delta\left(G_i(\A, \mathcal{B})\right)=\infty$. Then, by Theorem \ref{theo:duque-inf} we have that $\mathcal{A}$ is not nilpotent. 
\end{example}

\begin{figure}[h!]
    \centering
    \begin{tikzpicture}

\draw[thick,directed] (-2,0) -- (-1,0);
\draw[thick,directed] (-2,0) -- (-3,0);
\draw[thick,directed] (0,0) -- (-1,0);
\draw[thick,directed] (0,0) -- (1,0);
\draw[thick,directed] (2,0) -- (1,0);
\draw[thick,directed] (2,0) -- (3,0);
\draw[thick,directed] (4,0) -- (3,0);
\draw[thick,directed] (4,0) -- (5,0);
\draw[thick,directed] (6,0) -- (5,0);
\draw[thick,directed] (6,0) -- (7,0);
\draw[thick,directed] (8,0) -- (7,0);
\draw[thick,directed] (8,0) -- (9,0);
\draw[dashed,gray] (9,0) -- (9.7,0);
\draw[thick,directed] (-2,0) -- (-2,1);
\draw[thick,directed] (0,0) -- (0,1);
\draw[thick,directed] (0,1) -- (0,2);
\draw[thick,directed] (2,0) -- (2,1);
\draw[thick,directed] (2,1) -- (2,2);
\draw[thick,directed] (2,2) -- (2,3);
\draw[thick,directed] (4,0) -- (4,1);
\draw[thick,directed] (4,1) -- (4,2);
\draw[thick,directed] (4,2) -- (4,3);
\draw[thick,directed] (4,3) -- (4,4);
\draw[thick,directed] (6,0) -- (6,1);
\draw[thick,directed] (6,1) -- (6,2);
\draw[thick,directed] (6,2) -- (6,3);
\draw[thick,directed] (6,3) -- (6,4);
\draw[thick,directed] (6,4) -- (6,5);
\draw[thick,directed] (8,0) -- (8,1);
\draw[thick,directed] (8,1) -- (8,2);
\draw[thick,directed] (8,2) -- (8,3);
\draw[thick,directed] (8,3) -- (8,4);
\draw[thick,directed] (8,4) -- (8,5);
\draw[thick,directed] (8,5) -- (8,6);
\filldraw [black] (0,2) circle (2pt); 
\filldraw [black] (-2,0) circle (2pt); 
\filldraw [black] (-3,0) circle (2pt); 
\filldraw [black] (-2,1) circle (2pt); 
\filldraw [black] (2,1) circle (2pt); 
\filldraw [black] (2,2) circle (2pt); 
\filldraw [black] (2,3) circle (2pt); 
\filldraw [black] (-1,0) circle (2pt); 
\filldraw [black] (0,0) circle (2pt);    \node at (-2,-0.3) {\footnotesize $i_1$};     
\filldraw [black] (1,0) circle (2pt);     \node at (0,-0.3) {\footnotesize $i_2$}; 
\filldraw [black] (2,0) circle (2pt);     \node at (2,-0.3) {\footnotesize $i_3$}; 
\filldraw [black] (3,0) circle (2pt);    
\filldraw [black] (4,0) circle (2pt);     \node at (4,-0.3) {\footnotesize $i_4$}; 
\filldraw [black] (5,0) circle (2pt);
\filldraw [black] (6,0) circle (2pt);     \node at (6,-0.3) {\footnotesize $i_5$}; 
\filldraw [black] (7,0) circle (2pt);     \node at (8,-0.3) {\footnotesize $i_6$}; 
\filldraw [black] (0,1) circle (2pt);   
\filldraw [black] (4,1) circle (2pt);
\filldraw [black] (4,2) circle (2pt);
\filldraw [black] (4,3) circle (2pt);
\filldraw [black] (4,4) circle (2pt);
\filldraw [black] (6,1) circle (2pt);
\filldraw [black] (6,2) circle (2pt);
\filldraw [black] (6,3) circle (2pt);
\filldraw [black] (6,4) circle (2pt);
\filldraw [black] (6,5) circle (2pt);
\filldraw [black] (8,0) circle (2pt);
\filldraw [black] (8,1) circle (2pt);
\filldraw [black] (8,2) circle (2pt);
\filldraw [black] (8,3) circle (2pt);
\filldraw [black] (8,4) circle (2pt);
\filldraw [black] (8,5) circle (2pt);
\filldraw [black] (8,6) circle (2pt);
    \end{tikzpicture}
    \caption{Associated digraph for a Hilbert evolution algebra for which, given the algebra multiplication, the associated digraph satisfies $\sup_{i\in \mathbb{N}} \delta\left(G_i(\A,\mathcal{E})\right)=\infty$.}
    \label{fig:diam-go-inf}
\end{figure}
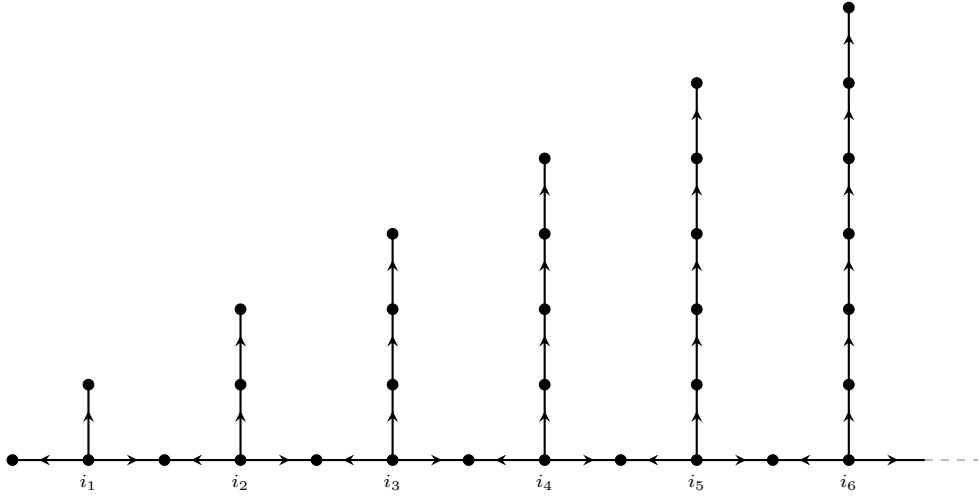





\bigskip
\begin{example}[Not nil and no oriented cycles]
Consider the Hilbert evolution algebra of Example \ref{exa:nonLF}, and let
$v=\sum_{i=1}^{\infty}v_i e_i$, with $v_i\neq 0$ for all $i\in\mathbb N$. Then,
$$v^2=\sum_{i=1}^{\infty}v_i^2 e_i^2=\sum_{i=2}^{\infty}\alpha_i e_i,$$
where $\alpha_2=v_1^2 c_{21}$, and $\alpha_i=v_1^2c_{i1}+v_{i-1}^2$ for $i\geq 3$. By induction on $k$, we can show that for any $k\in\mathbb{N}\setminus\{1\}$ there exist non-zero constants $\{\beta_i\}_{i\geq k}$ such that
$$v^{k}=\sum_{i=k}^{\infty}\beta_i e_i \neq 0.$$
Therefore $\A$ is not nil. Note that in this case, there are no oriented cycles in $G(\A, \mathcal{B})$ neither, see Figure \ref{fig:nonLF}. 
\end{example}

The previous examples suggest the following useful result.
\bigskip
\begin{lemma}
\label{lemma:finite diam implies sink}
Let $\A$ be a Hilbert evolution algebra with a natural orthonormal basis $\mathcal{B}=\{e_i\}_{i\in\mathbb{N}}$. Suppose then that there are no oriented cycles in $G(\A,\mathcal{B})$ and for some $i\in \mathbb{N}$ we have $\delta\left(G_{i}(\A,\mathcal{B})\right)<\infty$.
Then there exists a sink in $G_{i}(\A,\mathcal{B})$.
\end{lemma}
\begin{proof}
Let $\gamma:=\{i,\ldots,j\}$ be a  (directed) path in $G_i(\A,\mathcal{B})$ with length $l_i:=\delta\left(G_{i}(\A,\mathcal{B})\right)$. 
Let us prove that $j$ is a sink in $G_{i}(\A,\mathcal{B})$.
Suppose that $j$ is not a sink, then there exist some $k\in \mathbb{N}, k\notin\gamma$, such that $\omega_{jk}\neq 0$, because there are no oriented cycles in $G(\A, \mathcal{B})$.
Hence  $\{i,\ldots,j, k\}$ is a path of length $l_i+1$ which is impossible.
\end{proof}

\bigskip
\begin{theorem}\label{theo:duque-inf2}
Let $\A$ be a Hilbert evolution algebra with a  natural orthonormal basis $\mathcal{B}=\{e_i\}_{i\in\mathbb{N}}$. Then, the following conditions are equivalent:
\begin{enumerate}[label=(\roman*)]
\item \label{item:duque-inf2 --1}
$\A$ is nil.
\item \label{item:duque-inf2 --2}
There are no oriented cycles in $G(\A, \mathcal{B})$, and $ \delta\left(G_i(\A, \mathcal{B})\right)<\infty$, for any $i\in \mathbb{N}$. 
\item \label{item:duque-inf2 --3}
The weighted adjacency operator $\Omega:\ell^2(\mathbb{N})\longrightarrow\ell^2(\mathbb{N})$ can be represented by an infinite matrix $\{\omega_{ij}\}_{i,j\in \mathbb{N}}$ with strictly lower triangular form. 
\end{enumerate}
\end{theorem}

\begin{proof}
\ref{item:duque-inf2 --1} $\Rightarrow$ \ref{item:duque-inf2 --2}. Assume that $\A$ is nil. The proof of that there are no oriented cycles in $G(\A,\mathcal{B})$ use the same argument of \cite[Theorem 3.4]{Elduque/Labra/2015}. Now we shall prove, by contradiction, that $ \delta\left(G_i(\A, \mathcal{B})\right)<\infty$, for any $i\in \mathbb{N}$.  Indeed, note that if there exists $i_0\in \mathbb{N}$ such that $ \delta\left(G_{i_0}(\A,\mathcal{B})\right)=\infty$, then there exists a ray starting in $i_0$, say $\gamma:=\{i_0,i_1,i_2, \ldots\}$, such that, if we let 
$$v=\sum_{\ell=0}^{\infty}\alpha_{\ell} e_{i_{\ell}}\in \A,$$
with $\alpha_\ell \neq 0$ for any $\ell$, then $v^{n}\neq 0$ for any $n\geq 1$. Indeed, note that
$$v^2 = \sum_{\ell=0}^{\infty}\alpha_{\ell}^2 e_{i_{\ell}}^2 = \sum_{\ell=1}^{\infty}\alpha_{\ell}^2 e_{i_{\ell}} + \sum_{\ell=1}^{\infty}\left\{\sum_{k_{\ell} \in D(i_{\ell-1})\setminus \{i_{\ell}\}} \alpha_{\ell}^2 e_{k_\ell}\right\},$$
which, by induction, implies that for any $n\geq 1$: 
\begin{equation}\label{eq:locura}
v^n = \sum_{\ell = n-1}^{\infty} \alpha_{\ell}^n e_{i_{\ell}} + u_{n}\neq 0,
\end{equation}
where $u_n \in \overline{\span}\{e_i:i\in D^{n}(\gamma)\setminus\{i_1,\ldots,i_n\}\}$. Note that from some $n$ we could have $u_n=0$. Therefore we conclude from \eqref{eq:locura} that $\A$ is not nil, which contradicts our assumption. 


\ref{item:duque-inf2 --2} $\Rightarrow$ \ref{item:duque-inf2 --1}. Assume that $\A$ is not nil. Then, there exist $v\in\A$ such that $v^n\neq 0$, for any $n\in\mathbb{N}$. Let us write $v=\sum_{k}\alpha_k e_k$ and define the set $U:=\{k\in\mathbb{N}: \alpha_k\neq 0\}$ and $I:=\overline{\span}\{e_k:k\in U\}$.
Thus, for any $n\in\mathbb{N}$, we have
\begin{equation}\nonumber
v^n\in I^{<n>} \subset\overline{\span}\{e_k:k \in D^{n-1}(U)\}
\end{equation}
where we use the Proposition \ref{prop:seila}. 
Hence $D^{n-1}(U)\neq\emptyset$, for all $n\in\mathbb{N}$. This implies the existence of an oriented cycle or a ray in $G(\A,\mathcal{B})$.
Indeed, if 
$$D^n(U)\cap D^m(U)\neq \emptyset$$
for some pair $n,m$, say $n<m$, it follows the existence of an oriented cycle $\gamma:=\{i_0,\ldots, i_k,\ldots, i_0\}$ where the starting vertex is some $i_0\in D^n(U)\cap D^m(U)$ and $i_k\in D^{k}(U)$, for $k$ such that $n<k<m$. On the other hand, if 
$$D^n(U)\cap D^m(U)= \emptyset$$
for all $n,m\in \mathbb{N}$ we can construct a ray $\gamma:=\{i_0,i_1,\ldots,i_n,\ldots\}$ taking vertices such that $i_0\in U$, $i_n\in D^n(U)$ for $n\geq 1$.

\ref{item:duque-inf2 --2} $\Rightarrow$ \ref{item:duque-inf2 --3}. Note that the weighted adjacency operator $\Omega:\ell^2(\mathbb{N})\longrightarrow\ell^2(\mathbb{N})$ 
is densely defined, thus can be represented by an infinite matrix $\{\omega_{ij}\}_{i,j\in \mathbb{N}}$.
By Lemma \ref{lemma:finite diam implies sink} the graph have a sink
in some $j_1$, i.e., $e_{j_1}^2=0$, i.e., $\omega_{j_1k}=0$ for all $k\in\mathbb{N}$. 
Thus, we can reorder the vertex set by assigning $1$ to $j_1$, then $\omega_{1k}=0$ for all $k\in\mathbb{N}$ 
and we have $e_1^2=0$ and the first row of the matrix only has zeros.
Now, the subgraph with vertex set $\{2,3,\ldots\}$ also satisfies the hypothesis of  
Lemma \ref{lemma:finite diam implies sink}, which implies that the subgraph have 
a sink in some $j_2$, i.e., $\omega_{j_2k}=0$ for all $k\geq 2$. 
Again, we reorder the vertex set by assigning $2$ to $j_2$, then $e_2^2=\omega_{21}e_1$ 
and we have that the second row of the matrix has the first term $\omega_{21}$
and the rest are zero.
Repeating the process, we can reorder the vertex set to have 
$e_3^2=\omega_{31}e_1+ \omega_{32}e_2$ and we have that the third row of the matrix has 
the first and second terms $\omega_{31}, \omega_{32}$ and the rest are zero.
By induction we can see that we can reorder the vertex set to have 
$$e_n^2=\omega_{n1}e_1+\cdots+ \omega_{n(n-1)}e_{n-1}$$ 
and we have that the nth row of the 
matrix has the first $n-1$ terms $\omega_{n1},\ldots,\omega_{n(n-1)}$ and the rest are zero.
Thus, the operator $\Omega$ can be represented by an infinite matrix with strictly lower triangular form.

\ref{item:duque-inf2 --3}$\Rightarrow$ \ref{item:duque-inf2 --2}. Suppose that the infinite matrix $\{\omega_{ij}\}_{i,j\in \mathbb{N}}$ has a strictly lower triangular form.
Then, for the corresponding  natural orthonormal basis $\{e_i\}_{i\in\mathbb N}$, we have
$e_1^2=0$ and
$$e_n^2=\omega_{n1}e_1+\cdots+ \omega_{n(n-1)}e_{n-1}, \,\, \text{ for all } n\geq 1.$$ 
This shows directly that,
\begin{equation} \label{eq:descendants of n}
D^1(n)\subset \{1,2,\ldots, n-1\}, \,\,\text{ for all } n\geq 1.
\end{equation}
From this, induction on $n$ shows that
$$D^m(n)=D(n-(m-1))\subset\{1,2,\ldots, n-m\},\,\,\text{ for } 1\leq m\leq n.$$
This, and $e_1^2=0$ implies $D^m(n)=\emptyset$ if $m>n$.
Now consider $G_n(\A,\mathcal{B})=(V_n,E_n)$, then we have
\begin{equation} \label{eq:strong component of n}
\{e_k\in\A: k\in V_n\}\subset\{e_1,e_2,\ldots, e_{n-1}\},
\end{equation}
for all $n\in\mathbb{N}$.
From \eqref{eq:strong component of n} we obtain that
$$\delta\left(G_n(\A, \mathcal{B})\right)\leq n-1<\infty,$$ for any $n\in \mathbb{N}$.   
Now, let us prove that are no oriented cycles in $G(\A, \mathcal{E})$.
Suppose there exists an oriented cycle $\gamma:=\{i_1,i_2,\ldots, i_m\}$ for some indices $i_k$
and some $m\in\mathbb{N}$.
Note that we can reorder the set of vertices to have
$$i_1<i_k, \text{ for } 2\leq k\leq m.$$
Now, by Equation \eqref{eq:descendants of n} we  know that $D^1(i_1)=\{1,2,\ldots, i_1-1\}$,
thus, $i_k\notin D(i_1)$ for $2\leq k\leq m$.
That is, $\gamma$ can not be an oriented cycle.
\end{proof}

\section{Acknowledgements}

This research is part of the research project ``Algebras of Evolution in Separable Hilbert Spaces'' (Res. R/10 218-2022 UNPSJB). A part of this research was conducted during the visits of S.V. to the Universidade Federal de Pernambuco (UFPE), P.M.R. to the Universidade Federal do ABC (UFABC), and P.C. to the UFPE. The authors are grateful to these institutions for their hospitality and support. Part of this work was supported by Fundação de Amparo à Ciência e Tecnologia do Estado de Pernambuco - FACEPE (Grant APQ-1341-1.02/22) and Fundação de Amparo à Pesquisa do Estado de São Paulo - FAPESP (Grant 2017/10555-0).  The authors would like to express their gratitude to Ancykutty Joseph for sharing her work on infinite digraphs. They would also like to extend their appreciation to the referee, whose meticulous reading of the manuscript and insightful comments contributed to the improvement of this paper.



\begin{thebibliography}{99}

\bibitem{Cabrera/Siles/Velasco}
Y. Cabrera Casado, M.  Siles, M.V. Velasco, Evolution algebras of arbitrary dimension and their decompositions, Linear Algebra Appl. 495 (2016) 122-162.  
\url{https://doi.org/10.1016/j.laa.2016.01.007}.

\bibitem{PMPY} 
Y. Cabrera Casado, P. Cadavid, M.L. Rodi\~no Montoya, P.M. Rodriguez,  On the characterization of the space of derivations in evolution algebras, Annali di Matematica. 200 (2021) 737–755.  \url{https://doi.org/10.1007/s10231-020-01012-2}.

\bibitem{YPT} 
Y. Cabrera Casado, P. Cadavid, T. Reis,  Derivations and loops of some evolution algebras, Rev. Real Acad. Cienc. Exactas Fis. Nat. Ser. A-Mat. 117 (2023) 119.  \url{https://doi.org/10.1007/s13398-023-01446-2}.

\bibitem{PMP}
P. Cadavid, M.L. Rodi\~no Montoya, P.M. Rodriguez, The connection between evolution algebras, random walks, and graphs,  J. Algebra Appl. 19 (2020) 2050023. \url{https://doi.org/10.1142/S0219498820500231}.

\bibitem{PMPT} 
P. Cadavid, M.L. Rodi\~no Montoya, P.M. Rodriguez, On the isomorphisms between evolution algebras of graphs and random walks,  Linear Multilinear Algebra. 69 (2021) 1858-1877.  \url{https://doi.org/10.1080/03081087.2019.1645807}.

\bibitem{PMP2} 
P. Cadavid, M.L. Rodi\~no Montoya, P.M. Rodriguez, Characterization theorems for the space of derivations of evolution algebras associated to graphs, Linear Multilinear Algebra. 68 (2020) 1340-1354.  \url{https://doi.org/10.1080/03081087.2018.1541962}.

\bibitem{camacho/gomez/omirov/turdibaev/2013} L.M. Camacho, J.R. G\'omez, B.A. Omirov, R.M. Turdibaev, Some properties of evolution algebras, Bull. Korean Math. Soc. 50 (2013) 1481-1494.  \url{https://doi.org/10.4134/BKMS.2013.50.5.1481}.

\bibitem{camacho/gomez/omirov/turdibaev/LM2013}  
L.M. Camacho, J.R. G\'omez, B.A. Omirov, R.M. Turdibaev, The derivations of some evolution algebras,  Linear Multilinear Algebra. 61 (2013) 309–322.  \url{https://doi.org/10.1080/03081087.2012.678342}.


\bibitem{casas-ladra-rozikov}
J.M. Casas, M. Ladra, U.A. Rozikov, A chain of evolution algebras, Linear Algebra Appl. 435 (2011) 852-870.  \url{https://doi.org/10.1016/j.laa.2011.02.012}.

\bibitem{ceballos}
M. Ceballos, New advances on (pseudo)digraphs and evolution algebras, Comp. Appl. Math. 41, (2022) 148. \url{https://doi.org/10.1007/s40314-022-01858-7}.


\bibitem{DOR}
A. Dzhumadil’daev, B.A. Omirov, U.A. Rozikov, Constrained evolution algebras and dynamical systems of a bisexual population, 
Linear Algebra Appl. 496 (2016) 351–380.  \url{https://doi.org/10.1016/j.laa.2016.01.048}.



\bibitem{Elduque/Labra/2015}
A. Elduque, A. Labra, Evolution algebras and graphs, J. Algebra Appl. 14 (2015) 1550103.  \url{https://doi.org/10.1142/S0219498815501030}.

\bibitem{Elduque/Labra/2016}
A. Elduque, A. Labra, Evolution algebras, automorphisms, and graphs, Linear Multilinear Algebra. 69 (2021) 331-342.  \url{https://doi.org/10.1080/03081087.2019.1598931}.



\bibitem{Exner}
G.R. Exner, I.B. Jung, E.Y. Lee, On weighted adjacency operators associated to directed graphs, Filomat. 31 (2017) 4085-4104.  \url{https://doi.org/10.2298/FIL1713085E}.


\bibitem{Fujii} 
J.I. Fujii, Spectrum and Entropy for Infinite Directed Graphs,  Structural Analysis of Complex Networks. (2011) 105-136.
\url{http://dx.doi.org/10.1007/978-0-8176-4789-6_5}.



\bibitem{Halmos}  
P.R. Halmos,  A Hilbert space problem book, second ed., Springer New York, NY, 1982. (Graduate Texts in Mathematics, 19).
\url{http://dx.doi.org/10.1007/978-1-4684-9330-6}.

\bibitem{ancykutty0} 
A. Joseph, Logical numbering of acyclic infinite digraphs, J. Tri. Math. Soc. 3, 2001, 21-28. 

\bibitem{anckutty} 
A. Joseph, On infinite graphs and related matrices. (2005). Ph.D. Thesis. Cochin University of Science and Technology, Kerala, India.


\bibitem{Mukhamedov-Qaralleh}    
F. Mukhamedov, I. Qaralleh, Entropy Treatment of Evolution Algebras,  Entropy. 24 (2022) 595.  \url{https://doi.org/10.3390/e24050595}.



\bibitem{Paniello-EC}
I. Paniello,  Evolution coalgebras, Linear Multilinear Algebra. 67 (2018) 1539-1553.  \url{https://doi.org/10.1080/03081087.2018.1460795}

\bibitem{Paniello-BECCP} 
I. Paniello, Backwards inheritance in evolution coalgebras of chicken populations,  J. Algebra Appl. (2023). \url{https://doi.org/10.1142/S0219498824502396}


\bibitem{Paniello-IEOGC}  
I. Paniello In-evolution operators in genetic coalgebras, Linear Algebra Appl. 614 (2021) 197-207. \url{https://doi.org/10.1016/j.laa.2020.03.032}



\bibitem{Paniello-MEA}
I. Paniello, Markov evolution algebras, Linear Multilinear Algebra. 70 (2022) 4633-4653.  \url{https://doi.org/10.1080/03081087.2021.1893636}.



\bibitem{Qaralleh-Mukhamedov}
I. Qaralleh, F. Mukhamedov, Volterra evolution algebras and their graphs, Linear Multilinear Algebra. 69 (2021) 2228-2244.
 \url{https://doi.org/10.1080/03081087.2019.1664387}.

\bibitem{rozikov-velasco} 
U.A. Rozikov, M.V. Velasco,  Discrete-time dynamical system and an evolution algebra of mosquito population, J. Math. Biol. 78 (2019) 1225–1244. \url{https://doi.org/10.1007/s00285-018-1307-x}.

\bibitem{Schmudgen} 
K. Schmüdgen, Unbounded self-adjoint operators on Hilbert space, first ed., Springer Dordrecht, 2012. (Graduate Texts in Mathematics, 265).
\url{http://dx.doi.org/10.1007/978-94-007-4753-1}.
 
 

\bibitem{tian}
J. P. Tian, Evolution algebras and their applications, first ed., Springer Berlin, Heidelberg, 2008. (Lecture Notes in Mathematics, 1921).
\url{http://dx.doi.org/10.1007/978-3-540-74284-5}.


\bibitem{tv}
J.P. Tian, P. Vojtechovsky,  Mathematical concepts of evolution algebras in non-Mendelian genetics, Quasigroups Related Systems. 14 (2006) 111-122. 




\bibitem{vidal-IJPAM}  
S.J. Vidal, P. Cadavid, P.M. Rodriguez,  Hilbert evolution algebras and its connection with discrete-time Markov chains, Indian J. Pure Appl. Math. 54 (2023) 883–894.   \url{https://doi.org/10.1007/s13226-022-00304-y}.


\bibitem{vidal-SMJ} 
S.J. Vidal, P. Cadavid, P.M. Rodriguez, On Hilbert evolution algebras of a graph, Sib.  Math. J. 63 (2022) 995-1011.
 \url{https://doi.org/10.1134/S0037446622050184}.
 
\bibitem{Zhevlakov} 
K.A. Zhevlakov, A.M. Slin’ko, I.P. Shestakov,  A.I. Shirshov, Rings that are Nearly Associative, Academic Press, New York, 1982. (Pure and Applied Mathematics, 104)

\end{thebibliography}
\end{document}